\documentclass{article}

\usepackage{amsmath, amsthm, amsfonts, amssymb,mathrsfs}
\usepackage{graphicx}
\usepackage{caption}
\usepackage[]{units}
\usepackage{hyperref}
\usepackage{todonotes}

\newtheorem{thm}{Theorem}[section]
\newtheorem{prop}[thm]{Proposition}
\newtheorem{lem}[thm]{Lemma}
\newtheorem{cor}[thm]{Corollary}
\theoremstyle{definition}
\newtheorem{rem}[thm]{Remark}
\newtheorem{ex}[thm]{Example}

\graphicspath{{figures/}}

\DeclareMathOperator{\sgn}{sgn}
\DeclareMathOperator{\arctanh}{arctanh}
\newcommand{\RR}{\mathbb{R}}
\newcommand{\ZZ}{\mathbb{Z}}
\newcommand{\EE}{\mathbb{R}}
\newcommand{\HH}{{H}}

\newcommand{\CC}{\mathcal{C}}
\newcommand{\AAA}{\mathcal{A}}

\newcommand{\eps}{\varepsilon}
\newcommand{\la}{\langle}
\newcommand{\ra}{\rangle}
\newcommand{\dd}{\partial}
\newcommand{\Ksing}{\mathcal{K}_{\mathrm{sing}}}
\newcommand{\Ssing}{\Sigma_{\mathrm{sing}}}

\newcommand{\Imm}{\mathrm{Im}}
\newcommand{\spann}{\mathrm{span}}

\newcommand{\eq}[1]{\begin{align} #1 \end{align}}
\newcommand{\eqq}[1]{\begin{align}\begin{split}#1\end{split}\end{align}}
\newcommand{\al}[1]{\begin{align*} #1 \end{align*}}

\DeclareFontFamily{U}{mathx}{\hyphenchar\font45} 
\DeclareFontShape{U}{mathx}{m}{n}{<-> mathx10}{}
\DeclareSymbolFont{mathx}{U}{mathx}{m}{n}
\DeclareMathAccent{\widebar}{0}{mathx}{"73}

\begin{document}


\title{Embeddedness of timelike maximal surfaces in $(1+2)$-Minkowski space}
\author{E.\ Adam Paxton\footnote{Department of Physics, University of Oxford, Parks Road, Oxford, UK}}

\maketitle
\begin{abstract}
We prove that if $\phi \colon \RR^2 \to \EE^{1+2}$ is a smooth proper timelike immersion with vanishing mean curvature, then necessarily $\phi$ is an embedding, and every compact subset of $\phi(\RR^2)$ is a smooth graph. It follows that if one evolves any smooth self-intersecting spacelike curve (or any planar spacelike curve whose unit tangent vector spans a closed semi-circle) so as to trace a timelike surface of vanishing mean curvature in $\EE^{1+2}$, then the evolving surface will either fail to remain timelike, or it will fail to remain smooth. We show that, even allowing for null points, such a Cauchy evolution must undergo a scalar curvature blow-up---where the blow-up is with respect to an $L^1\-L^\infty$ norm---and thus the evolving surface will be $C^2$ inextendible beyond singular time. In addition we study the continuity of the unit tangent for the evolution of a self-intersecting curve in isothermal gauge, which defines a well-known evolution beyond singular time.
\end{abstract}
\tableofcontents
\listoffigures

\section{Introduction \& statement of main results}
The study of minimal surfaces in Euclidean space $\EE^3$ has a long history, and many interesting examples of complete minimal surfaces in $\EE^3$ are known. On the other hand, many beautiful theorems have demonstrated that minimal surfaces in $\RR^3$ exhibit a certain rigidity. For example, Bernstein's theorem states that any complete minimal surface in $\RR^3$ which is a graph, must be a plane. In this article we consider the timelike maximal surfaces in Minkowski space $\EE^{1+2}$, where the picture is quite different. 

If $\phi \colon M^2 \to \RR^{1+2}$ is a smooth proper timelike immersion then $\phi^0$ is a Morse function and it follows that $M^2$ is diffeomorphic to either $\RR^2$ or $S^1 \times \RR$. In appropriate coordinates, the mean curvature of $\phi(M^2)$ is hyperbolic, and by solving a Cauchy problem with sufficiently ``small'' initial data, it is possible to construct smooth proper timelike maximal immersions $\phi\colon \RR^2 \to \RR^{1+2}$ such that $\phi(\RR^2)$ is a smooth graph close to a timelike plane (see Lindblad \cite{lindblad}). This clearly contrasts with Bernstein's theorem in $\EE^3$ (for more stability results in higher dimensions and higher codimensions see Allen, Anderson \& Isenberg \cite{AAI}, Brendle \cite{brendle}, Donninger, Krieger, Szeftel \& Wong \cite{DKSW}, as well as \cite{lindblad}). 

On the other hand, given suitably ``large'' data, the Cauchy evolution for a timelike maximal surface will develop singularity in finite time (see e.g.\ Bellettini, Hoppe, Novaga \& Orlandi \cite{BHNO}, Eggers \& Hoppe \cite{EH}, Kibble \& Turok \cite{KT}, and Nguyen \& Tian \cite{NT}. See also Bahouri, Marachli \& Perelman \cite{BMP}, Eggers, Hoppe, Hynek \& Suramlishvili \cite{EHHS} and Wong \cite{WWong} for results in higher dimensions). Nguyen \& Tian proved: {there exists no smooth proper timelike immersion $\phi\colon S^1\times \RR \to \EE^{1+2}$ with vanishing mean curvature} \cite{NT}. Thus the Cauchy evolution of any closed curve will form singularity in finite time, and every smooth proper timelike maximal immersion in $\RR^{1+2}$ is of the form $\phi \colon \RR^2 \to \RR^{1+2}$. 

In this article we will be concerned with the geometry of timelike maximal immersions $\phi \colon \RR^2 \to \RR^{1+2}$ and the corresponding Cauchy evolution for open curves. Our first result is:
\begin{thm}\label{maintheorem}
Let $\phi \colon \RR^2 \to \EE^{1+2}$ be a smooth, proper, timelike immersion with vanishing mean curvature. Then $\phi$ is an embedding. Moreover, for each compact subset $K\subseteq \phi(\RR^2)$, there is a timelike plane $P\subseteq \EE^{1+2}$ such that $K$ is a smooth graph over $P$.
\end{thm} 
\begin{rem}
The restriction to compact subsets in Theorem \ref{maintheorem} cannot be relaxed, and there exist examples (see Subsection \ref{graphsection}) of smooth proper timelike maximal embeddings $\phi \colon \RR^2 \to \EE^{1+2}$ such that $\phi(\RR^2)$ is not a graph. 
\end{rem}
\begin{rem}
\sloppy If $\phi \colon \RR^2 \to \RR^{1+2}$ is a smooth proper timelike maximal immersion, then in terms of a spacelike unit normal 
$N \colon \phi(\RR^2) \to S^{1+1} = \big\{ (\sinh 
\varphi, \cos \vartheta \cosh \varphi , \sin \vartheta \cosh \varphi ) \colon (\vartheta,\varphi)\in \RR^2 \big\}\subseteq\RR^{1+2}$ 
Theorem \ref{maintheorem} states that for every compact subset $K \subseteq \phi(\RR^2)$, $N(K)$ is contained in an open hemi-hyperboloid $S^{1+1}_+ = \big\{ (\sinh 
\varphi, \cos \vartheta \cosh \varphi , \sin \vartheta \cosh \varphi ) \colon  (\vartheta,\varphi)\in (\vartheta_0-\frac{\pi}{2}, \vartheta_0 + \frac{\pi}{2}) \times \RR \big\}  \subseteq \RR^{1+2}$ for some $\vartheta_0 \in \RR$ (which is a hemi-sphere with respect to the Minkowski metric). This may be compared with the counterpart in the Riemannian setting. For example, it is well-known that for a complete minimal surface in $\EE^3$ the image of the unit normal is either a single point, or it omits at most 4 points in the sphere $S^2$. 
\end{rem}
\begin{rem}
As a crucial step in the proof of Theorem \ref{maintheorem}, we adapt an argument of Belletini, Hoppe, Novaga \& Orlandi \cite{BHNO} to construct a global system of isothermal coordinates on an immersed timelike maximal surface (another such construction of global isothermal coordinates may be found in \cite[Chap. 7]{weinsteinbook}). This is a non-trivial result as, in stark contrast with the Riemannian setting, there exist infinitely many possible conformal structures of simply connected Lorentzian surfaces (Kulkarni \cite{kulkarni}).
\end{rem}
The coordinate $x^0$ on $\RR^{1+2}$ is a time-function, and we now turn to a Cauchy problem for timelike maximal surfaces in $\EE^{1+2}$. Let $\CC\colon \RR \to \{x^0=0\}\subseteq \RR^{1+2}$ be a smooth proper immersion and let $V$ be a smooth future-directed timelike vector field along $\CC$. We say $\phi\colon \RR \times [-T,T] \to \EE^{1+2}$, $\phi(s,t)=(t,\gamma^1(s,t),\gamma^2(s,t))$, is a smooth timelike Cauchy evolution for $(\CC,V)$ if $\phi$ is a smooth proper timelike immersion with vanishing mean curvature such that $\phi(\cdot, 0)=\CC$ and $V$ is tangent to $\mathrm{Im}(\phi)$ along $\CC$. For a given smooth initial data $(\CC,V)$ let 
\al{T_*=\sup\{T\geq0\colon &\text{there exists a smooth timelike Cauchy evolution} \\ 
&\phi \colon  \RR \times [-T,T] \to \EE^{1+2} \hskip3pt \text{for} \hskip2pt (\CC,V)  \}.}
It may be shown that $T_*>0$ under mild assumptions on the initial data $(\CC,V)$ (see e.g.\ Corollary \ref{ste}) and from Theorem \ref{maintheorem} it may be seen to follow that if the image $U_0(\CC)$ of the unit tangent vector $U_0$ along $\CC$ contains a closed semi-circle (for example, if $\CC$ is a self-intersecting curve) then $T_*<\infty$. However, our proof of Theorem \ref{maintheorem} is by contradiction, and thus does not shed any light upon the nature of singularity at time $T_*$. It is natural to ask whether one can define a smooth, or $C^k$ for some $k$, extension of the surface beyond singular time, possibly by allowing for null points.

It is well-known that singular behaviour necessarily involves the maximal surface failing to remain timelike at the time $T_*$ (i.e.\ the hyperbolicity degenerates), see Jerrard, Novaga \& Orlandi \cite[Theorem 3.1]{jerrardetal}. Eggers \& Hoppe \cite{EH} studied singularity formation in a self-similar regime, and observed a blow up of curvature of the spatial cross-sections at the singular time $T_*$. Nguyen \& Tian observed that, provided the 2\textsuperscript{nd} order term in a certain Taylor expansion is non-vanishing, then the limit curve at singular time $T_*$ will look locally like a $C^{1,\nicefrac{1}{3}}$ graph \cite[Remark 2.6]{NT}. Since the 2\textsuperscript{nd} order term is expected to be generically non-vanishing, one thus expects a blow up of curvature at the singular time generically. We prove:
\begin{thm}\label{noextensionlemma}
Let $\eps>0$, $(s_0,t_0)\in \RR^2$, and $\phi \colon (s_0-\eps,s_0+\eps) \times (t_0-\eps, t_0] \to \EE^{1+2}$ be a $C^1$ immersion of the form $\phi(s,t)=(t,\gamma^1(s,t),\gamma^2(s,t))$, such that $\phi|_{(s_0-\eps,s_0+\eps)\times(t_0-\eps, t_0)}$ is $C^2$ and timelike with bounded mean curvature. Suppose that $\phi$ is null at the point $(s_0,t_0)$, i.e.\ $\mathrm{Im}(d\phi_{(s_0,t_0)})$ is a null plane in $\RR^{1+2}$. Then the curvature of the (planar) curves $\gamma(\cdot, t)$ blows up as $t\nearrow t_0$, and $\phi$ is not $C^2$. 
\end{thm}
\begin{rem}
In fact, we deduce Theorem \ref{noextensionlemma} from a stronger result which gives the precise rate of curvature blow-up in an $L^1L^\infty$ norm. Moreover, whilst Theorem \ref{noextensionlemma} assumes the case that the limit curve at singularity formation is $C^1$, this blow-up rate holds without assuming any structure of the singularity. See Proposition \ref{regularityloss} for details.
\end{rem}
Theorem \ref{noextensionlemma} rules out (in all cases) the possibility of a $C^2$ causal extension of the Cauchy evolution beyond singular time. However, one may still ask whether there exists a $C^1$ causal extension. A complete answer to this question, independent of gauge, is currently out of reach. Nonetheless, we will proceed to consider one well-known extension beyond singular time: by solving the maximal surface equations globally in isothermal gauge (a construction somewhat analogous to the Weierstrass representation for minimal surfaces in $\EE^3$) \cite[Chap. 8]{weinsteinbook}, \cite[Chap. 7]{string}. 

Let us briefly recall the method of isothermal gauge. Since we are now concerned with the prospect of less regular maximal surfaces, it is natural to consider less regular initial data $(\CC,V)$ (other weak notions of solution have been considered by Belletini, Novaga \& Orlandi \cite{lorentzian_varifolds} and Brenier \cite{brenier}). Let $\CC\colon \RR \to \{x^0=0\}\subseteq \EE^{1+2}$ be a $C^k$ proper immersion, $k\geq1$, and let $V$ be a $C^{k-1}$ future-directed timelike vector field along $\CC$. One may construct a proper $C^k$ map $\phi\colon \RR^2 \to \EE^{1+2}$ of the form $\phi(s,t)=(t,\gamma(s,t))$, where $\gamma$ satisfies (in the weak sense if $k=1$) the system of equations $\la \gamma_s,\gamma_t\ra=0$, $|\gamma_s|^2+|\gamma_t|^2=1$, $\gamma_{tt}-\gamma_{ss}=0$, such that $\mathrm{Im}(\phi(\cdot,0))=\mathrm{Im}(\CC)$ and $\Sigma=\phi(\RR^2)$ is tangent to $V$ along $\CC$. $\phi$ defines a $C^k$ timelike maximal immersion on $\RR^2\setminus\Ksing$ where $\Ksing = \{ (s,t) \colon \gamma_s(s,t)=0\}$ and $\Sigma$ gives a $C^k$ timelike maximal surface away from $\Ssing=\phi(\Ksing)$. For every $p\in \Ssing$ either $\Sigma$ fails to be a $C^1$ surface in a neighbourhood of $p$ or $\Sigma$ is a $C^1$ surface in a neighbourhood of $p$ but is null at $p$. See Section \ref{isothermalgaugesection} for more details. 

From Theorems \ref{maintheorem} and \ref{noextensionlemma} it follows that if $U_0(\CC)$ contains a closed semi-circle, then $\Sigma$ cannot be a $C^2$ immersed surface (see Corollary \ref{noC2cor}). There are (non-generic) cases however where $\Sigma$ is $C^1$ immersed. Indeed, in Example \ref{regularexample1} we present a curve $\CC$ for which $U_0(\CC)$ is exactly a closed semi-circle and show that an evolution by isothermal gauge of $\CC$ yields a $C^1$ embedded surface which is a smooth timelike maximal surface away from a pair of null half-lines. This surface contains non-graphical compact sets (compare Theorem \ref{maintheorem}). It turns out however that the situation of Example \ref{regularexample1} is borderline. We prove:
\begin{thm}\label{gaugetheorem}
Let $\phi\colon\RR^2\to\EE^{1+2}$, $\phi(s,t)=(t,\gamma(s,t))$, be a $C^1$ evolution for a maximal surface by isothermal gauge, as described above, and write $U_0\colon \RR \to S^1$ for the unit tangent vector along the initial curve $\gamma(\cdot,0)$. Suppose that $\mathrm{Im}(U_0)$ contains an arc of length $>\pi$ (for example if $\gamma(\cdot,0)$ is self-intersecting). Then there exists a time $t_*\in\RR$ such that: either $\mathrm{Im}(\gamma(\cdot,t_*))$ is not a $C^1$ immersed curve; or $\mathrm{Im}(\gamma(\cdot,t_*))$ is a $C^1$ immersed curve, but the spatial unit tangent $U(\cdot,t_*) = \nicefrac{\gamma_s(\cdot,t_*)}{|\gamma_s(\cdot,t_*)|}$ (defined only on the set $\{s \colon \gamma_s(s,t_*)\neq0\}$) admits no extension to a continuous unit tangent vector field along $\gamma(\cdot,t_*)$.
\end{thm}
In most cases, the discontinuity of the spatial unit tangent corresponds to the curve $\gamma(\cdot,t_*)$ failing to be $C^1$. Eggers \& Hoppe \cite{EH} introduced the swallowtail singularity, whereby the first singularity is a $C^{1,\nicefrac{1}{3}}$ curve which immediately splits off into a twin pair of travelling cusps. This picture was shown to be (in some sense) generic, for sufficiently regular initial data, by Nguyen \& Tian \cite[Section 3]{NT}. There exist, however, non-generic cases whereby the discontinuity of the unit tangent does not imply a regular cusp, and it is possible that the unit tangent admits no continuous extension along $\gamma(\cdot,t_*)$, whilst $\mathrm{Im}(\gamma(\cdot,t_*))$ is a $C^1$ immersed curve, see Example \ref{sheetingexample}. Although we have no example where such a degenerate situation occurs whilst the surface $\phi(\RR^2)$ remains $C^1$, we don't rule this out.

Finally, we note that Theorem \ref{maintheorem} fails for timelike maximal surfaces in $\EE^{1+n}$ for $n\geq3$. Nguyen \& Tian gave an example of a smooth, proper, timelike maximal immersion $\phi \colon S^1\times \RR \to \EE^{1+3}$ \cite[Appendix]{NT}, and it was conjectured that generic closed curves do not evolve to singularities in higher codimension. This conjecture was confirmed by Jerrard, Novaga \& Orlandi \cite{jerrardetal} who showed that when $n\geq4$, generic closed curves with generic initial velocity will evolve to a globally regular surface, whilst in the borderline case $n=3$ there are distinct, non-empty open sets of initial data leading to both regular surfaces and singular surfaces respectively. It is simple to see how the example of \cite[Appendix]{NT} may be adapted to give a smooth proper self-intersecting timelike maximal immersion $\phi \colon \RR^2 \to \EE^{1+3}$ and it would be interesting to obtain similar results to \cite{jerrardetal} for open curves.

{\bf Structure of the paper.} In Section \ref{preliminaries} we introduce the timelike maximal surface equations, and give a construction of global isothermal coordinates on any properly immersed timelike maximal surface (Lemma \ref{lemma2}). In Section \ref{embeddednesssection} we prove Theorem \ref{maintheorem} and give examples of both graphical and non-graphical timelike maximal surfaces. In Section \ref{c2inextendibilitysection} we prove Theorem \ref{noextensionlemma}, and we discuss in a bit more detail the rate of curvature blow-up (see Proposition \ref{regularityloss} and Example \ref{shrinkingcircle}). Section \ref{singularevolution} is then devoted to analysis in isothermal gauge. In Subsection \ref{isothermalgaugesection} we recall the isothermal gauge construction and gather some known results. In Subsection \ref{singularsetanalysis} we give further analysis of the solution by isothermal gauge. In particular we present local and global existence results which are notable in that they require no decay on the initial data at infinity (Corollary \ref{ste} and Remark \ref{globalexistenceremark}) and we give localized singularity statements to complement Theorem \ref{maintheorem} (Proposition \ref{keyprop} and Corollary \ref{noC2cor}). In Subsection \ref{c1extendibleexamples} we give examples illustrating some possible (non-generic) singular behaviours, including $C^1$ properly embedded surfaces containing non-graphical compact sets which are smooth timelike maximal surfaces away from a pair of null half-lines (Example \ref{regularexample1}), and $C^1$ properly embedded graphical (but not $C^1$ graphical) periodic surfaces which are smooth timelike maximal surfaces away from a discrete lattice of null points (Example \ref{regularexample2}). In Subsection \ref{gaugesection} we give the proof of Theorem \ref{gaugetheorem}, and we also present some more examples of possible non-generic singular behaviours (Examples \ref{cuspswitcherexample} and \ref{sheetingexample}).

{\bf Acknowledgement.} I would like to thank my supervisor, Luc Nguyen, for being so generous with his time, and for many insightful comments. This work was completed with the support of the
Engineering and Physical Sciences Research
Council [EP/L015811/1].

\section{Preliminaries}\label{preliminaries}
In this section we will first give a brief recap of the maximal surface equations. We will then present an adaptation of the construction of global isothermal coordinates which was given by Belletini, Hoppe, Novaga \& Orlandi in \cite{BHNO}, for a spatially compact timelike maximal surface, to the spatially non-compact case. We note that another construction of global isothermal coordinates is given in \cite[Chapter 7]{weinsteinbook}.
\subsection{Maximal surface equations}\label{maxsurfaceequations}
Let $x=(x^0,x^1,x^2)$ denote standard (i.e.\ {\it inertial}) coordinates on $\EE^{1+2}$, so that the Minkowski metric is $\eta = -(dx^0)^2 + (dx^1)^2 + (dx^2)^2$. Let $\Omega\subseteq \RR^2$ be an open subset and $\phi \colon \Omega \to \EE^{1+2}$ be a $C^1$ immersion. We write $\phi^\alpha = x^\alpha \circ \phi$ for the expression of $\phi$ in coordinates, $\alpha=0,1,2$, and denote the image of $\phi$ by $\Sigma=\mathrm{Im}(\phi)$. The metric $g$ induced by $\phi$ is the bilinear form $g_p \colon T_p\RR^2 \times T_p\RR^2 \to \RR$ given by $g_p(X,Y)=\eta(d\phi_p(X), d\phi_p(Y))$. 

For each $p\in \RR^2$, recall that $\phi$ is {{timelike}} at $p$ if $\det(g_p)<0$, $\phi$ is {{null}} at $p$ if $\det(g_p)=0$, $\phi$ is {spacelike} if $\det(g_p)>0$, and $\phi$ is {{causal}} at $p$ if $\phi$ is either timelike or null at $p$. We say that $\phi$ is timelike (resp.\ causal) if it is timelike (resp.\ causal) at every point $p$. In the case that $\phi$ is timelike at $p$, there exists a choice of unit spacelike normal vector $N(p)$, and we have a direct sum decomposition of the tangent space which is orthogonal with respect to $\eta$,
\al{ T_{\phi(p)}\EE^{1+2} = \mathrm{span}\{N(p)\} \oplus T_{\phi(p)} \Sigma.}

Let $(s,t)$ denote coordinates on $\Omega\subseteq \RR^2$. For every compact subset $V\subseteq \Omega$, define the {area} of $\phi(V)$ as
\al{ \AAA\left[ \phi ; V\right] = \int_V \sqrt{|\det(g(s,t))|} ds dt.}
The area of $\phi(V)$ is independent of the choice of coordinates $(s,t)$ on $V$. The Euler-Lagrange equations associated to the area functional $\AAA$ are
\eq{\label{gaussweingarten} \frac{1}{\sqrt{|\det g|}} \dd_i \left( \sqrt{|\det g|} g^{ij} \partial_j \phi^\alpha \right)=0,}
having adopted the summation convention. We say that a $C^1$ immersion $\phi$ is {\it maximal} if it satisfies \eqref{gaussweingarten} in the weak sense. When $\phi$ is a $C^2$ timelike immersion, \eqref{gaussweingarten} is equivalent to $\HH(\phi)=0$, where $\HH$ is the mean-curvature vector of $\phi(\Omega)$.

\eqref{gaussweingarten} is independent of the choice of coordinates, so if $\phi\colon \RR^2 \to \RR^{1+2}$ is a smooth solution to \eqref{gaussweingarten} and $\psi \colon \RR^2 \to \RR^2$ is a smooth diffeomorphism, then $\phi'=\phi \circ \psi$ also solves \eqref{gaussweingarten}. \eqref{gaussweingarten} is also invariant under rescaling of $\EE^{1+2}$, as well as the isometries of $\EE^{1+2}$. For a timelike immersion, with respect to a system of isothermal coordinates, \eqref{gaussweingarten} reduces to the wave equation 
\al{\phi_{tt}-\phi_{ss}=0.} 

\subsection{Construction of isothermal coordinates}

\begin{lem}\label{lemma1}
Let $\phi \colon \RR^2 \to \EE^{1+2}$ be a smooth, proper, timelike immersion. Then there exists a smooth diffeomorphism $\psi \colon \RR^2 \to \RR^2$ such that $\phi'=\phi \circ \psi$ is of the form $\phi'(s,t) = (t, \gamma^1(s,t), \gamma^2(s,t))$ where $\gamma=(\gamma^1,\gamma^2)$ satisfies $|\gamma_s|^2=1$.
\end{lem}

\begin{lem}[Existence of global isothermal coordinates]\label{lemma2}
Let $\phi \colon \RR^2 \to \EE^{1+2}$ be a smooth, proper, timelike immersion with vanishing mean curvature. Then there exists a smooth diffeomorphism $\psi \colon \RR^2 \to \RR^2$ such that $\phi'=\phi \circ \psi$ is of the form $\phi'(s,t) = (t, \gamma^1(s,t), \gamma^2(s,t))$ where $\gamma=(\gamma^1,\gamma^2)$ satisfies
\eq{\label{orthogcondition1}\la \gamma_s , \gamma_t \ra &= 0 \\
\label{orthogcondition2}|\gamma_s|^2 + |\gamma_t|^2 &=1 \\ 
\label{waveequation}\gamma_{tt} - \gamma_{ss} &=0.}
\end{lem}
\begin{proof}[Proof of Lemma \ref{lemma1}.]
The proof is a standard argument exploiting the fact that $\phi^0$ is a Morse function. Let $\phi \colon \RR^2 \to \EE^{1+2}$ be a smooth, proper, timelike immersion. For each $t\in \mathrm{Im}(\phi^0)$ write
\al{ C_t = \{ (y^1, y^2) \in \RR^2 \colon \phi^0(y^1,y^2) = t\} .}
Since $\phi$ is timelike, $\phi^0$ can have no critical points. Thus $C_t$ is a smooth submanifold of $\RR^2$ for all $t\in \mathrm{Im}(\phi^0)$ by the implicit function theorem. 

Let $g=\phi^*\eta$ be the induced Lorentzian metric on $\RR^2$, and let $X=\nabla_g\phi^0$, which is a smooth, nowhere-vanishing vector field on $\RR^2$. $\phi(C_t)=\mathrm{Im}(\phi) \cap \{ x^0=t \}$ is spacelike, so with respect to $g$, the submanifolds $C_t$ are spacelike, and thus $X$ is a timelike vector field orthogonal to the submanifolds $C_t$. 

Define $T = \frac{1}{g(X,X)} X$, and consider the flow of $T$. Let $p\in\RR^2$, and let $\xi_p \colon (a,b) \to \RR^2$, be the smooth, inextendible integral curve of $T$ through $p$, so $\frac{d \xi_p}{d s} (s)= T(\xi_p(s))$ and $\xi_p(0)=p$. Then $\frac{d}{d s} \left( \phi^0(\xi_p(s)) \right) = (d \phi^0 )_{\xi_p(s)}(T(\xi_p(s))) =1$ and so
\eq{\label{tequation}\phi^0(\xi_p(s))=\phi^0(p) + s.}
We claim that $b=\infty$ and $a=-\infty$. Indeed, suppose we had $b<\infty$. Since the curve $\xi_p$ is timelike, and by \eqref{tequation}, then $\phi(\xi_p([0,b)))$ would lie in the intersection of the time slab $0\leq t \leq b$ with the future-directed light cone with vertex at the point $\phi(p)$, i.e.\ those points $(x^0,x^1,x^2)\in\RR^3$ such that
\al{(x^1-\phi^1(p))^2 + (x^2 - \phi^2(p))^2 \leq (x^0 - \phi^0(p))^2 \\
\phi^0(p) \leq x^0 \leq \phi^0(p)+b, }
which is a compact set. Since $\phi$ is a proper map, it would follow that the curve $\xi_p([0,b))$ would lie in a compact set. As $T$ is smooth, it would then follow that $\xi_p$ could then be smoothly extended up to $s=b$, contradicting inextendibility of $\xi_p$. So $b=\infty$ and similarly $a=-\infty$. 

From \eqref{tequation}, it is seen that the flow $p\mapsto \xi_p(t)$ maps $C_0$ diffeomorphically onto $C_t$ for each $t$, thus we have shown $\mathrm{Im}(\phi^0)=\RR$, and we have a foliation of $\RR^2$ given by smooth curves $C_t$ for $t\in \RR$. We claim that each $C_t$ is connected. Indeed, for $p,q\in C_0$, let $\omega \colon [0,1] \to \RR^2$ be a continuous path with $\omega(0)=p$, $\omega(1)=q$. Define $\hat{\omega}(s)=\xi_{\omega(s)}(-\phi^0(\omega(s)))$, so $\hat{\omega}(s)\in C_0$ for all $s\in [0,1]$ by \eqref{tequation}, and $\hat{\omega}$ is a continuous path with $\hat{\omega}(0)=p$ and $\hat{\omega}(1)=q$. Thus $C_0$ and hence each $C_t$ is connected.

Let $C_0$ be given some parameterisation as $C_0(s)$ for $s\in(-\infty, \infty)$, and define $\psi \colon \RR^2 \to \RR^2$ by
\al{ \psi(s,t) = \xi_{C_0(s)}(t).} 
By the group property of the flow, it is seen that $\psi$ gives a bijection. Standard results on smooth dependence on initial conditions for ODE show that $\psi$ gives a smooth map, and since $T$ is nowhere vanishing and orthogonal to $C_0$ we have $\det(d\psi)(s,0)\neq 0$ and so it follows $\det(d\psi)(s,t)\neq 0$ for all $(s,t)\in \RR^2$, see eg. \cite[Chapter 1]{codandlev}. Thus $\psi$ is a diffeomorphism, and we have $\phi'=\psi \circ \phi$ satisfies $\phi'(s,t) = (t, \gamma^1(s,t), \gamma^2(s,t))$. Finally, since $\phi$ is proper, it follows that $|\gamma(s,t)|^2 \to \infty$ as $s \to \pm \infty$ for each $t$. Thus we may pass to an arclength reparameterisation for each $t$ to ensure the condition $|\gamma_s(s,t)|^2 =1$. 
\end{proof}
\begin{proof}[Proof of Lemma \ref{lemma2}.]
By applying Lemma \ref{lemma1}, we may assume that $\phi$ is of the form 
\al{\phi(s,t)=(t,\gamma^1(s,t), \gamma^2(s,t))} 
where $|\gamma_s|^2=1$. Since $\phi$ is timelike, we have the bound $|\gamma_t|^2 <1$. 

Now, let $s'=s'(s,t)$, $t'=t$ denote a smooth coordinate change, with $\frac{\dd s'}{\dd s} >0$, and set $\gamma'(s', t')=\gamma(s,t)$. We will choose these new coordinates so that 
\eq{\label{orthogcondition}\la \gamma'_{s'}, \gamma'_{t'} \ra =0.} 
By the chain rule: 
\eq{\label{chain1}\gamma'_{s'}&= \left( \frac{\dd s'}{\dd s} \right)^{-1} \gamma_s \\
\label{chain2} \quad\quad \gamma'_{t'}&= -\left(\frac{\dd s' }{\dd s} \right)^{-1}\left( \frac{\dd s'}{\dd t} \right) \gamma_s + \gamma_t.} 
Substituting expressions \eqref{chain1} and \eqref{chain2}, and observing $|\gamma_s|^2=1$, we see that \eqref{orthogcondition} will be satisfied provided
\eq{\label{transporteq}\frac{\dd s'}{\dd t} - \la \gamma_s , \gamma_t \ra \frac{\dd s'}{\dd s} =0.}
This is a linear transport equation, and may be solved by the method of characteristics. The solution $s'$ is constant along characteristic curves $(s(t),t)$, where the $s(t)$ are solutions to
\eq{\label{characteristicequation}\dot{s}(t)=-\la \gamma_s(s(t),t), \gamma_t(s(t),t) \ra.} 
Since the right hand side of \eqref{characteristicequation} is smooth, and since we have the a-priori bound 
\eq{\label{aprioribound}|\la \gamma_s, \gamma_t \ra|<1,} 
smooth solutions to \eqref{characteristicequation} exist for all $t\in\RR$, and for each $(s_0,t_0)$, there exists a unique characteristic through $(s_0,t_0)$ which crosses through the line $\{t=0\}$ precisely once. Thus for any smooth function $\rho\colon \RR \to \RR$, there is a unique smooth solution $s'$ to \eqref{transporteq} satisfying the Cauchy data 
\al{s'(s,0) = \rho(s).} 
The choice of Cauchy data $\rho$ will be fixed later. For now, observe that that the condition $\frac{\dd s'}{\dd s} > 0$ is equivalent to 
\eq{\label{diffcondition1}\dot{\rho}(s)>0,}
and, by the uniform bound on the characteristic speed \eqref{aprioribound}, we have $s'(s,t) \to \pm \infty$ as $s\to \pm \infty$ for each $t$ provided 
\eq{\label{diffcondition2} \rho(s) \to \pm \infty} 
as $s \to \pm \infty$. A smooth diffeomorphism $\psi \colon \RR^2 \to \RR^2$ is thus well defined by $\psi^{-1}(s,t)=(s'(s,t),t)$ provided $\rho$ is chosen so that \eqref{diffcondition1} and \eqref{diffcondition2} hold. 

We have now verified \eqref{orthogcondition} (which is \eqref{orthogcondition1} in the $(s',t')$ coordinates), and we proceed to show that $\rho$ may be selected satisfying \eqref{diffcondition1} and \eqref{diffcondition2}, so as to ensure \eqref{orthogcondition2} and \eqref{waveequation}. From \eqref{gaussweingarten}, the maximal surface equations read
\eq{\label{GW1} \dd_i( \sqrt{|\det(g)|} g^{i 2}) &=0 \\
\label{GW2}\dd_i( \sqrt{|\det(g)|} g^{ij} \dd_j \gamma) &=0.}
Since the metric in the new coordinates is
\al{ g(s',t')= |\gamma'_{ s'}|^2 ds'^2 + (-1 + |\gamma'_{t'}|^2) dt'^2}
the first of these reads
\al{\dd_{t'}\sqrt{\frac{|\gamma'_{s'}|^2}{1- |\gamma'_{t'}|^2}}=0}
which is equivalent to ${|\gamma'_{s'}(s',t')|^2}=K(s')^2({1- |\gamma'_{t'}(s',t')|^2})$. Thus the condition
\al{|\gamma'_{s'}|^2 + |\gamma_{t'}|^2 =1}
will follow provided $\rho(s)$ is chosen such that $|\gamma'_{s'}(s',0)|^2+|\gamma'_{t'}(s',0)|^2=1$ (i.e $K(s')^2=1$). From \eqref{chain1}, \eqref{chain2} and \eqref{transporteq} we have
\al{|\gamma'_{s'}(s',0)|^2+|\gamma'_{t'}(s',0)|^2&=|\dot{\rho}(s)^{-1} \gamma_s(s,0)|^2 \\
&\hskip25pt +|\gamma_t(s,0)-\la \gamma_s(s,0), \gamma_t(s,0) \ra \gamma_s(s,0)|^2 \\
&=\dot{\rho}(s)^{-2} + |\gamma_t(s,0)|^2 - \la \gamma_s(s,0), \gamma_t(s,0) \ra^2}
which equals $1$ provided
\al{\dot{\rho}(s)&=(1-|\gamma_t(s,0)|^2 + \la \gamma_s(s,0), \gamma_t(s,0) \ra^2)^{-1/2} \\
&=|\det(g(s,0))|^{-1/2}.}
Since $\phi$ is timelike, this ensures \eqref{diffcondition1} and moreover by the bound 
\al{0<|\det(g(s,t))|\leq 1} 
we see
\al{ \rho(s)= \int_*^s (|\det(g(s,0))|)^{-1/2} \,ds \to \pm \infty}
as $s \to \pm\infty$, which is \eqref{diffcondition2}. We have ensured \eqref{orthogcondition1} and \eqref{orthogcondition2}, and as the metric now reads 
\al{g(s',t')=|\gamma'_{s'}(s',t')|^2\left( ds'^2- dt'^2\right),} 
the equation $\gamma'_{t't'} -\gamma'_{s's'}=0$ follows from \eqref{GW2}. This completes the proof.
\end{proof}

\section{Embeddedness of maximal surfaces}\label{embeddednesssection}
In this section we give the proof of Theorem \ref{maintheorem}, as well as examples of both graphical and non-graphical timelike maximal surfaces. The latter examples show that the restriction to compact subsets in Theorem \ref{maintheorem} cannot be relaxed in general.
\subsection{Proof of Theorem \ref{maintheorem}}
In light of Lemma \ref{lemma2}, consider a smooth, proper, timelike immersion $\phi \colon \RR^2\to \EE^{1+2}$ of the form
\eq{\label{specialform}\phi(s,t) = (t, \gamma^1(s,t), \gamma^2(s,t))}
where $\gamma=(\gamma^1, \gamma^2)$ satisfies
\eq{\label{orthogcondition11} \la \gamma_s, \gamma_t\ra &=0 \\
\label{orthogcondition12} |\gamma_s|^2+|\gamma_t|^2 &=1 \\
\label{waveequation2} \gamma_{tt}-\gamma_{ss} &= 0.}
Define
\eq{\label{aplusminus} a_\pm(s)= \gamma_t(s,0) \pm \gamma_s(s,0),}
so that $|a_\pm(s)|^2 =1$ by \eqref{orthogcondition11}, \eqref{orthogcondition12}. $a_\pm$ give the spatial directions of the outgoing and incoming null tangent vectors to $\phi(\RR^2)$ along the initial curve $\phi(\cdot, 0)$. The following Lemma shows that the images of the outgoing and incoming null directions must be disjoint for a smooth, timelike, properly immersed maximal surface. 
\begin{lem}\label{dalembertlemma}
Let $\phi \colon \RR^2 \to \EE^{1+2}$ be a smooth, proper, timelike immersion of the form \eqref{specialform}, where $\gamma$ satisfies \eqref{orthogcondition11}--\eqref{waveequation2} and define $a_\pm$ by \eqref{aplusminus}. Then $a_+(\xi) \neq a_-(\eta)$ for all $\xi, \eta \in \RR$. 
\end{lem}
\begin{proof}
Since $\gamma$ satisfies the wave equation \eqref{waveequation2}, we have d'Alembert's formula
\eq{\label{dalembert} \gamma(s,t)=\frac{1}{2} \left( \gamma(s+t,0) + \gamma(s-t,0) + \int_{s-t}^{s+t} \gamma_t(\xi,0) \,d\xi \right).}
Differentiating gives
\eqq{\label{dalembertdifferentiated}\gamma_s(s,t)&= \frac{1}{2} \left( \gamma_s(s+t,0) + \gamma_s(s-t,0) + \gamma_t(s+t,0) - \gamma_t(s-t,0) \right) \\
&= \frac{1}{2}\left( a_+(s+t) - a_-(s-t) \right).}
Since $\phi$ is an immersion, $\gamma_s(s,t)\neq 0$ for all $(s,t)\in \RR^2$, and thus $a_+(\xi)\neq a_-(\eta)$ for all $\xi,\eta\in \RR$.
\end{proof}
\begin{lem}\label{geometrylemma}
Let $M>0$ and let $a_\pm \colon [-M,M] \to \RR^2$ be smooth functions satisfying $|a_\pm|^2=1$ and $a_+(\xi)\neq a_-(\eta)$ for all $\xi,\eta \in [-M,M]$. Then there exists $\omega \in \RR^2$, $|\omega|^2=1$, such that
\eq{\la a_+(\xi) - a_-(\eta), \omega \ra > 0 }
for all $\xi,\eta \in [-M,M]$.
\end{lem}
\begin{proof}
$A=\mathrm{Im}(a_+)$ is a non-empty, connected, closed, proper subset of $S^1$, so we may write 
\al{A=\{(\cos \alpha, \sin \alpha) \colon \alpha\in [\alpha_1,\alpha_2]\}.}
Defining $\omega= (\cos \frac{\alpha_1+\alpha_2}{2}, \sin \frac{\alpha_1+\alpha_2}{2})$, it follows from trigonometry that $\la a, \omega \ra > \la b, \omega \ra$ for all $a\in A$, $b \in S^1\setminus A$. Since it is assumed $\mathrm{Im}(a_-) \subseteq S^1\setminus A$, the claim is proved.
\end{proof}
We now have the tools to hand to prove Theorem \ref{maintheorem}.
\begin{proof}[Proof of Theorem \ref{maintheorem}]
Let $\phi \colon \RR^2 \to \EE^{1+2}$ be a smooth, proper, timelike immersion with vanishing mean curvature. By Lemma \ref{lemma2}, we may take $\phi$ to be of the form $\phi(s,t)=(t,\gamma(s,t))$ where $\gamma$ satisfies \eqref{orthogcondition11}--\eqref{waveequation2}. 
 
Let $M>0$ and define the characteristic diamond
\eq{D_M=\left\{ (s,t) \colon |s|+|t| \leq M \right\} \subseteq \RR^2.}
To prove the theorem, we will show that $\phi|_{D_M}$ is injective and $\phi(D_M)$ is a smooth graph over a timelike plane $P_M$. Since $M$ is arbitrary, from this it will follow that $\phi$ is injective, and thus an embedding. Since $\phi$ is proper, given any compact subset $K \subseteq \phi(\RR^2)$, we may choose $M$ sufficiently large such that $K\subseteq \phi(D_M)$, so that $K$ will be a smooth graph over the plane $P_M$. 

Defining $a_\pm$ as in \eqref{aplusminus}, by Lemma \ref{dalembertlemma} we have that $a_+(\xi)\neq a_-(\eta)$ for all $\xi, \eta \in \RR$. So by Lemma \ref{geometrylemma} there exists $\omega_M \in \RR^2$, $|\omega_M|^2=1$, such that 
\al{\la a_+(\xi) - a_-(\eta) , \omega_M \ra >0}
for all $\xi,\eta\in[-M,M]$. From \eqref{dalembertdifferentiated}, it follows
\eq{\label{easyinequality} \la \gamma_s(s,t) ,\omega_M \ra = \frac{1}{2}\la a_+(s+t) - a_-(s-t), \omega_M \ra >0}
for all $(s,t) \in D_M$.

From \eqref{easyinequality} it is now routine to show that $\phi|_{D_M}$ is an embedding and there is a timelike plane $P_M\subseteq \EE^{1+2}$ such that $\phi(D_M)$ is a smooth graph over $P_M$, but we will go through the argument for completeness. Rotating coordinates on $\RR^{1+2}$ as necessary, we may assume for convenience that $\omega_M=(1,0)$. Then, in the new coordinates, keeping the same notation for the parameterisation, \eqref{easyinequality} reads
\eq{\label{easyinequality2} \gamma^1_s(s,t) >0}
for all $(s,t) \in D_M$. Let $P_M$ be the $x^0$--$x^1$ plane in these new coordinates.

Write $D'_M=\{(t,\gamma^1(s,t)) \colon (s,t)\in D_M\}\subseteq \RR^2$, and let $F\colon D_M \to D'_M$ be given by $F(s,t)=(t,\gamma^1(s,t))$. From \eqref{easyinequality2} it follows by monotonicity that $F$ is bijective, and moreover by the inverse function theorem that $F$ is a smooth diffeomorphism. Inverting $F$ as $F^{-1}(x^0,x^1)=(s(x^0,x^1), t(x^0,x^1))$ gives
\eqq{\label{graphform}\phi(D_M)&= \phi\circ F^{-1} (D'_M) \\
&= \left\{ \left(x^0, x^1, \gamma^2(s(x^0,x^1), t(x^0,x^1))\right) \colon (x^0,x^1) \in D'_M \right\} }
so we have shown $\phi(D_M)$ is a smooth graph over the $x^0$--$x^1$ plane. Moreover, it follows from \eqref{graphform} that $\phi \circ F^{-1} \colon  D'_M \to \EE^{1+2}$ is injective, so $\phi|_{D_M}$ is injective. This completes the proof.
\end{proof}
\subsection{Examples of graphical and non-graphical smooth properly embedded timelike maximal surfaces}\label{graphsection}
\begin{ex}[Smooth, properly embedded, graphical timelike maximal surfaces]\label{graphexample}
Let $f\colon \RR \to \RR$ be any smooth function, and let $G=\{(u,f(u)) \colon u\in \RR\} \subseteq \RR^2 $ be the graph of $f$. Let $c\colon \RR\to \RR^2$ be a smooth parameterisation of $G$ by arclength, so that $\mathrm{Im}(c)=G$ and $|\dot{c}(s)|=1$. Define $\gamma(s,t)=\frac{1}{2} \left( c(s+t) + c(s-t) \right)$ and $\phi \colon \RR^2 \to \EE^{1+2}$ by $\phi(s,t) = (t, \gamma(s,t))$. It may be checked that $\phi$ defines a smooth, proper, timelike embedding with vanishing mean curvature, and $\phi(\RR^2)$ is a smooth graph over the $x^0$--$x^1$ plane with $\phi(\RR^2)\cap\{x^0=0\}=G$.
\end{ex}
\begin{ex}[Smooth, properly embedded, doubly periodic graphical timelike maximal surfaces]
Let $f\colon \RR \to \RR$ be a smooth function such that $f(0)=0$, and $f(u)=f(u+1)$ for all $u\in \RR$ (i.e.\ $f$ is periodic). As in Example \ref{graphexample}, let $c \colon \RR \to \RR^2$ parametrize the graph of $f$ by arclength, and define $\gamma(s,t)=\frac{1}{2} \left( c(s+t) + c(s-t) \right)$ and $\phi \colon \RR^2 \to \EE^{1+2}$ by $\phi(s,t) = (t, \gamma(s,t))$.

Note that $c(s+L)=c(s)+(1,0)$, where $L$ is the length of one period of $f$. Necessarily $L\geq 1$ with equality if and only if $f\equiv 0$ (i.e.\ if and only if the graph of $f$ is a straight line). Then observe that $\phi(s+L,t)=\phi(s,t)+(0,1,0)$, and $\phi(s,t+L)=\phi(s,t)+(L,0,0)$. Thus defining $T\colon \EE^{1+2} \to \EE^{1+2}$ by $T(x^0,x^1,x^2)=(x^0+L, x^1,x^2)$ for a translation in time, and $R \colon \EE^{1+2} \to \EE^{1+2}$ by $R(x^0,x^1,x^2)=(x^0,x^1+1,x^2)$ for a translation in space, we see $\phi(\RR^2)$ is invariant under both $T$ and $R$. Thus $\phi(\RR^2)$ is periodic in the direction $(1,0,0)$ with period $L$, and periodic in the direction $(0,1,0)$ with period $1$. 

By acting on $\phi(\RR^2)$ by a combination of a rescaling and a Lorentz tranformation, it may be seen that, for any timelike vector $V\in \RR^{1+2}$, and spacelike vector $W\in \RR^{1+2}$ orthogonal to $V$, and for any pair of numbers $(a,b)$ with $a>b$, one may obtain smooth, non-planar, graphical timelike maximal surfaces which are periodic in the direction $V$ with period $a$, and periodic in the direction $W$ with period $b$.
\end{ex}
\begin{ex}[Smooth, properly embedded, non-graphical timelike maximal surfaces.]\label{nongraphexample}
Let $c\colon \RR \to \RR^2$ be a parametrisation of a smooth curve by arclength such that the following hold: 
\begin{enumerate}
\item $c(s)=(0,-s)$, for $s\in(-\infty,-1]$,
\item $\dot{c}^1(s)>0$ for $s\in(-1,\infty)$,
\item as $s\to \infty$, $\dot{c}(s) \to (0,1)$.
\end{enumerate}
See Figure \ref{grimreaper} for a rough illustration of such a curve. Every compact subset $K$ of $\mathrm{Im}(c)$ is a smooth graph, but $\mathrm{Im}(c)$ is not a smooth graph.
\begin{figure}[h]
\centering
 \captionsetup{width=.9\linewidth}
\includegraphics[width=.3\textwidth]{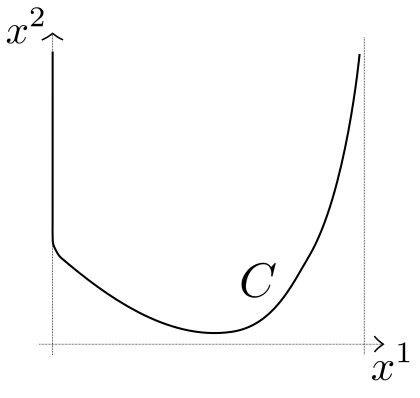}
\caption[A non-graphical curve for which every compact subset is a graph]{A smooth planar curve which is not a graph, for which every compact subset is a graph.}\label{grimreaper}
\end{figure}

Define $\gamma(s,t)=\frac{1}{2} \left( c(s+t) + c(s-t) \right)$ and $\phi \colon \RR^2 \to \EE^{1+2}$ by $\phi(s,t) = (t, \gamma(s,t))$. Then $\phi$ defines a smooth, proper, timelike embedding with vanishing mean curvature. For every compact subset $K\subseteq \phi(\RR^2)$, there is a timelike plane $P\subseteq \EE^{1+2}$ such that $K$ is a smooth graph over $P$, which is consistent with Theorem \ref{maintheorem}. We now claim that $\phi(\RR^2)$ is not a graph. To see this, observe that $\phi(s,t)=(t,0,-s)$ for $s\leq -1-|t|$, so $\phi(\RR^2)$ contains a closed quadrant $\bar{Q}=\{(t,0,-s)\colon s\leq -1-|t|\}$ of the plane $\{x^1=0\}$. For all $t\in \RR$, the curve $s\mapsto \phi(s,t)$ asymptotes to the plane $\{x^1=1\}$ as $s\to \infty$. It then follows that for every point $q$ in the interior of $\bar{Q}$, every straight line in $\RR^{1+2}$ through $q$ intersects $\phi(\RR^2)$ at at least 2 distinct points. Thus $\phi(\RR^2)$ is not a graph.

In this example, the image of the unit normal $N(\phi(\RR^2))$ is not contained in any open hemi-hyperboloid, but is contained in the union of an open hemi-hyperboloid with one connected component of its boundary.
\end{ex}

\section{$C^2$ inextendibility: Proof of Theorem \ref{noextensionlemma}}\label{c2inextendibilitysection}
For the rest of this article, we will be concerned with the question of whether it is possible to relax the notion of a maximal surface, either by allowing for surfaces which are $C^k$ for some $k\geq1$, or by allowing for null points (i.e.\ degenerate hyperbolicity), in such a way as to continue beyond singular time in a Cauchy evolution. 

Our first result in this direction will be that, if the evolution fails to remain timelike, then the maximal surface must fail to be $C^2$ immersed. In fact, we will deduce this from a broader observation which holds for more general evolutions of surfaces of only bounded mean curvature.
\begin{prop}\label{regularityloss}
Let $\Omega \subseteq \RR^2$ be an open bounded set such that for some $(s_0,t_0)\in \RR^2$ and some $\eps>0$, one has $\{s_0\}\times[t_0-\eps,t_0) \subseteq \Omega$ and $(s_0,t_0)\in \dd \Omega$. Let $\phi \colon \widebar{\Omega} \to \EE^{1+2}$ be a $C^1$ map such that $\phi|_\Omega$ is a $C^2$ timelike immersion, and such that $\phi$ is of the form $\phi(s,t)=(t,\gamma(s,t))$, where $\gamma$ satisfies $\la \gamma_s(s_0,t) ,\gamma_t (s_0,t) \ra =0$ for $t\in [t_0-\eps,t_0)$. Write $h$ for the mean curvature scalar of $\phi$, and $k(\cdot,t)$ for the curvature of the (planar) curves $\gamma(\cdot, t)$. Suppose $|h(s,t)|\leq C$ for $(s,t)\in \Omega$, and $|\gamma_t(s_0,t_0)|^2=1$ (so if $\phi$ is an immersion, then $\phi$ is null at $(s_0,t_0)$). Then
\eq{\label{curvatureintegral}\int_{t_0-\eps}^{t_0} |k(s_0,t)| dt = \infty.}
In particular, $\limsup_{t\nearrow t_0}|k(s_0,t)| = \infty$.
\end{prop} 
\begin{proof}
By taking $\eps$ sufficiently small, we may ensure that $|\gamma_t(s,t)|^2>0$ for $(s,t)\in \Omega \cap B_\eps(s_0,t_0)$. It may then be seen that a spacelike unit normal vector $N$ to $\phi(\Omega \cap B_\eps(s_0,t_0))$ is given along $\{s_0\}\times[t_0-\eps,t_0)$ by
\al{N(s_0,t)
	=  \frac{1}{(1-|\gamma_t(s_0,t)|^2)^{\nicefrac{1}{2}}}
		\left(\begin{array}{c}|\gamma_t(s_0,t)| \\ n(s_0,t) \end{array}\right)
 ,}
where 
\al{n(s_0,t)=\nicefrac{\gamma_t(s_0,t)}{|\gamma_t(s_0,t)|}} 
is a unit normal to the planar curve $\gamma(\cdot,t)$ at the point $s=s_0$.

The curvature of the cross sections $\gamma(\cdot,t)$ is given at $s=s_0$ by
\al{k(s_0,t)= \frac{\la\gamma_{ss}(s_0,t), n(s_0,t)\ra}{|\gamma_s(s_0,t)|^2}.}
Along $\{s_0\}\times[t_0-\eps,t_0)$, the components of the first fundamental form $E(s,t)ds^2+2F(s,t)dsdt+G(s,t)dt^2$ are calculated as
\al{ E(s_0,t)&= |\gamma_s(s_0,t)|^2 \\ 
F(s_0,t)&= \la\gamma_s(s_0,t), \gamma_t(s_0,t)\ra = 0 \\
G(s_0,t)&= -1+|\gamma_t(s_0,t)|^2,}
and the components of the second fundamental form $e(s,t)ds^2+2f(s,t)dsdt+g(s,t)dt^2$ are
\al{e(s_0,t)&= -\frac{\la \gamma_{ss}(s_0,t) , n(s_0,t) \ra}{(1-|\gamma_t(s_0,t)|^2)^{\nicefrac{1}{2}}} \\
f(s_0,t)&= -\frac{\la \gamma_{st}(s_0,t) , n(s_0,t) \ra}{(1-|\gamma_t(s_0,t)|^2)^{\nicefrac{1}{2}}} \\ 
g(s_0,t)&= -\frac{\la \gamma_{tt}(s_0,t) , n(s_0,t) \ra}{(1-|\gamma_t(s_0,t)|^2)^{\nicefrac{1}{2}}}.}
The mean curvature scalar is 
\eqq{\label{meancurvatureequation}h(s_0,t)&= \frac{e(s_0,t)}{E(s_0,t)} + \frac{g(s_0,t)}{G(s_0,t)} \\
 &= -\frac{\la \gamma_{ss}(s_0,t), n(s_0,t) \ra}{|\gamma_s(s_0,t)|^2(1-|\gamma_t(s_0,t)|^2)^{\nicefrac{1}{2}}}+ \frac{\la \gamma_{tt}(s_0,t), n(s_0,t) \ra }{(1-|\gamma_t(s_0,t)|^2)^{\nicefrac{3}{2}}},}
and rearranging \eqref{meancurvatureequation} gives the identity
\eq{\label{rearrangedcurvature}(1-|\gamma_t(s_0,t)|^2)^{\nicefrac{1}{2}} h(s_0,t) + k(s_0,t) = \frac{\la \gamma_{tt}(s_0,t) , n(s_0,t) \ra}{1-|\gamma_t(s_0,t)|^2}.}
Next we claim that
\eq{\label{integral}\int_{t_0-\eps}^{t_0} \frac{\la \gamma_{tt}(s_0,t) , n(s_0,t) \ra}{1-|\gamma_t(s_0,t)|^2} = \infty.}
To show \eqref{integral}, write $\mu(t)=|\gamma_t(s_0,t)|^2$, so that
\al{\frac{\la\gamma_{tt}(s_0,t) , n(s_0,t) \ra}{1-|\gamma_t(s_0,t)|^2}= \frac{\la\gamma_{tt}(s_0,t) , \gamma_t(s_0,t) \ra}{|\gamma_t(s_0,t)| (1-|\gamma_t(s_0,t)|^2)}=\frac{\frac{1}{2}\dot{\mu}(t)}{\mu(t)^{\nicefrac{1}{2}} (1-\mu(t))}.}
We have by assumption $\mu(t)\nearrow1$ as $t\nearrow t_0$, so
\al{\int_{t_0-\eps}^{t_0} \frac{\dot{\mu}(t)}{1-\mu(t)} dt = \int^{t_0}_{t_0-\eps} -\frac{d}{dt} (\log(1-\mu(t))) dt  = \infty}
from which \eqref{integral} follows. 

As $|h(s,t)|\leq C$, \eqref{curvatureintegral} then follows from \eqref{rearrangedcurvature} and \eqref{integral} and the Proposition is proved.
\end{proof}
\begin{ex}[Shrinking circle]\label{shrinkingcircle}
Define $\phi \colon S^1 \times (-\frac{\pi}{2},\frac{\pi}{2}) \to \EE^{1+2}$ by $\phi(s,t) = (t, \gamma(s,t))$, where 
\al{\gamma(s,t)= (\cos t\cos s, \cos t\sin s).} Then one may compute $h(s,t)=0$, and $\phi$ is a timelike maximal immersion. In addition, $\la\gamma_s,\gamma_t\ra =0$ (the parameterisation is orthogonal) and $|\gamma_t(s,t)|^2\nearrow 1$ as $t\nearrow \frac{\pi}{2}$. Observe $|k(s,t)|=|\cos t|^{-1}$, and $\int_0^{\frac{\pi}{2}} |k(s,t)| dt = \infty$ for all $s$, which is consistent with \eqref{curvatureintegral}. For this example, we may study the rate of blow-up in more detail. The element of arclength along $\gamma(\cdot,t)$ is $d\sigma(s) = |\cos t| ds $, thus for $p,q\in (1,\infty)$, one has 
\al{\|k\|_{L^q((0,\frac{\pi}{2});L^p(S^1))}&= \left(\int_0^{\frac{\pi}{2}} \left( \int_0^{2\pi} |k(s,t)|^p d\sigma(s) \right)^{\frac{q}{p}} dt \right)^{\frac{1}{q}} \\ 
&=(2\pi)^\frac{1}{p}\left(\int_0^\frac{\pi}{2} |\cos t|^\frac{q(1-p)}{p} dt\right)^\frac{1}{q}} 
and
\al{\|k\|_{L^q((0,\frac{\pi}{2});L^p(S^1))}<\infty \quad \text{if and only if} \quad \frac{1}{p} + \frac{1}{q} > 1.} 
\end{ex}
\sloppy The shrinking circle is $C^1$ inextendible beyond the singular time $\frac{\pi}{2}$ (in fact, the maximal extension of $\phi\left((S^1 \times (-\frac{\pi}{2},\frac{\pi}{2})\right)$ to a $C^0$ submanifold of $\RR^{1+2}$ is given by taking the closure of $\phi\left(S^1 \times (-\frac{\pi}{2},\frac{\pi}{2})\right)$ i.e.\ by attaching one point at $x^0=\frac{\pi}{2}$ and one at $x^0=-\frac{\pi}{2}$). In Subsection \ref{c1extendibleexamples}, we will see examples where the evolution is $C^2$ inextendible, but $C^1$ extendible.
\begin{proof}[Proof of Theorem \ref{noextensionlemma}]:
Let $\phi \colon (s_0-\eps, s_0+\eps) \times (t_0-\eps, t_0] \to \EE^{1+2}$, $\phi(s,t)=(t, \gamma^1(s,t), \gamma^2(s,t))$, be a $C^1$ immersion which is a $C^2$ timelike immersion with bounded mean curvature on $(s_0-\eps,s_0+ \eps) \times (t_0-\eps, t_0)$, and which is null at the point $(s_0,t_0)$. For a sufficiently small $\eps_0\in(0,\eps)$, let $r \colon [t_0-\eps_0, t_0] \to \RR$ be a solution to the terminal value problem 
\al{\dot{r}(t)= -\frac{\la \gamma_s(r(t), t) , \gamma_t(r(t),t) \ra}{|\gamma_s(r(t), t)|^2};\quad r(t_0)= s_0,}
which satisfies $|r(t)-s_0|<\frac{\eps}{2}$ for all $t\in [t_0-\eps_0, t_0]$ (such a solution exists by the Peano existence theorem). We have $r\in C^2([t_0-\eps_0,t_0)) \cap C^1([t_0-\eps_0,t_0])$.

Define $\Omega = (-\eps_0,\eps_0) \times (t_0-\eps_0,t_0)$. Let $\phi' \colon \widebar{\Omega} \to \EE^{1+2}$, $\phi'(s',t')=(t',\gamma'^1(s',t'),\gamma'^2(s',t'))$ where $\gamma'=(\gamma'^1, \gamma'^2)$ is given by $\gamma'(s',t') = \gamma(r(t')+s',t')$. Then $\phi'$ is a $C^1$ immersion which is a $C^2$ timelike immersion with bounded mean curvature on $\Omega$, and $\phi'(0,t_0)=\phi(s_0,t_0)$. By the chain rule,
\al{\la \gamma'_{s'}(s',t'), \gamma'_{t'}(s',t') \ra = \dot{r}(t') &|\gamma_s(r(t') + s', t')|^2  \\ 
&+ \la \gamma_s(r(t') +s',t'), \gamma_t(r(t')+s',t') \ra }
so by construction we have
\al{\la \gamma'_{s'}(0,t') , \gamma'_{t'}(0,t') \ra =0}
for $t'\in (t_0-\eps_0,t_0)$. As $\phi'$ is null at $(0,t_0)$, it may be seen that $|\gamma'_{t'}(0,t_0)|^2=1$. So since $|h(s',t')|\leq C$ for $(s',t')\in \Omega$, we see $\phi'$ satisfies the conditions for Proposition \ref{regularityloss}, so
\eq{\label{infinitecurvature}\limsup_{t'\nearrow t_0} |k(0,t')| = \infty}
where $k(\cdot,t')$ is the curvature of the planar cross sections $\gamma'(\cdot, t')$. Thus the curvatures of the curves $\gamma(\cdot,t)$ are not uniformly bounded for $t\in [t_0-\eps,t_0]$, so $\phi$ is not $C^2$.
\end{proof}

\section{Evolution beyond singular time by isothermal gauge}\label{singularevolution}
As is well documented in the physics literature, see e.g.\ \cite[Chap.\ 7]{string}, one global notion of Cauchy evolution, which defines a timelike maximal surface away from some possible singular set, may be given for arbitrary initial data by solving the maximal surface equations in isothermal gauge. In fact, we have already encountered this construction in Examples \ref{graphexample}--\ref{nongraphexample} and \ref{shrinkingcircle}. 

In Subsection \ref{isothermalgaugesection} we will recall how to evolve by isothermal gauge. In Subsection \ref{singularsetanalysis} we will prove some results on bounds for the singular set, including a criterion (in terms of only the initial curve) for determining whether the singular set is non-empty in some localized patch, as well as a result of short-time existence. In Subsection \ref{c1extendibleexamples}, we will present some examples whereby the evolution by isothermal gauge yields $C^1$ embedded surfaces which are non-graphical (these examples are interesting in light of Theorem \ref{maintheorem}). Finally, in Subsection \ref{gaugesection}, we will address the question of for which initial data sets the isothermal gauge yields a $C^1$ immersed surface, and prove Theorem \ref{gaugetheorem} which demonstrates an obstruction to constructing $C^1$ immersed surfaces by isothermal gauge which are not embedded.

\subsection{Evolution by isothermal gauge}\label{isothermalgaugesection}

Let $\CC\colon \RR \to \EE^{1+2}$, be a $C^k$, $k\geq 1$, proper immersion of the form
\eq{\label{cform}\CC(s)=(0,c(s))}
and let $V$ be a $C^{k-1}$, future-directed, timelike vector field along $\CC$. We refer to the pair $(\CC,V)$ as the initial data. 

We will construct a surface $\Sigma\subseteq\EE^{1+2}$ containing $\mathrm{Im}(\CC)$, with $V$ tangent to $\Sigma$ along $\mathrm{Im}(\CC)$, which is a $C^k$ immersed timelike maximal surface away from some (possibly empty) singular set. 

The prescription of the initial data $(\CC,V)$ is equivalent to a prescription of a curve $\CC$ and a continuous distribution of timelike tangent planes along $\CC$. By changing basis as necessary, we may thus assume $V$ is of the form 
\eq{\label{vequation}V(s) = (1, v(s))} 
where 
\eq{\label{initialdataorthog1}\la \dot{c}(s),v(s) \ra = 0} 
($c=(c^1,c^2)$, $v=(v^1,v^2)$). Since $V$ is timelike implies $|v(s)|<1$, we may then reparametrize the curve $\CC(s)$ to ensure the additional constraint 
\eq{\label{initialdataorthog2}|\dot{c}(s)|^2 + |v(s)|^2 =1}
holds. The pair $(\dot{\CC}(s),V(s))$ gives an orthonormal frame along the initial data, and the timelike planes $T_{\CC(s)}\Sigma =\mathrm{span}\{\dot{\CC}(s), V(s)\}$ are spanned by the null vectors 
\eq{\label{aplusminus'} A_\pm(s)=V(s)\pm \dot{\CC}(s)=\left(1,v(s)\pm\dot{c}(s)\right)=(1,a_\pm(s)).} 

Next, define a $C^k$ map $\phi \colon \RR^2 \to \EE^{1+2}$ by $\phi(s,t)=(t,\gamma^1(s,t),\gamma^2(s,t))$ where $\gamma=(\gamma^1,\gamma^2)$ is given by d'Alembert's formula
\eq{\label{dalembert2} \gamma(s,t) = \frac{1}{2} \left( c(s+t) + c(s-t) + \int_{s-t}^{s+t} v(\zeta) d\zeta \right) .}

\eqref{dalembert2} implies that
\eq{\label{wave}\gamma_{tt}-\gamma_{ss}&=0 \\
\label{initialconditions} \gamma(s,0)=c(s); \quad& \gamma_t(s,0)=v(s)}
with \eqref{wave} understood in the weak sense when $\gamma$ is not $C^2$. The isothermal gauge conditions 
\eq{\label{og1}\la \gamma_s(s,t), \gamma_t(s,t) \ra &=0 \\
\label{og2}|\gamma_s(s,t)|^2+|\gamma_t(s,t)|^2 &=1}
are satisfied for all $(s,t)\in \RR^2$ by \eqref{dalembert2}. We will call $\phi\colon \RR^2 \to \EE^{1+2}$ the evolution of $(\CC,V)$ by isothermal gauge. 

Write
\al{\Sigma&=\phi(\RR^2)}
and define the closed (possibly empty) singular set by
\eq{\label{ksing}\Ksing = \{ (s,t) \in \RR^2 \colon \gamma_s(s,t) =0\} }
so that $\phi$ gives a $C^k$ immersion on $\RR^2 \setminus \Ksing$. Then from \eqref{wave}, \eqref{og1}, \eqref{og2} we see that on $\RR^2 \setminus \Ksing$, $\phi$ defines a timelike, maximal immersion. Write
\eq{\label{ssing}\Ssing&= \phi(\Ksing).}
By construction $\Sigma \setminus \Ssing$ gives a $C^k$ timelike maximal immersed surface containing $\CC$ and tangent to the velocity field $V$ along $\CC$.

The following simple topological result shows that this is indeed a global evolution.
\begin{lem}\label{propernesslemma}
Let $\phi\colon \RR^2 \to \EE^{1+2}$, $\phi(s,t)=(t,\gamma(s,t))$ be an evolution by isothermal gauge for a $C^1\times C^0$ initial data $(\CC,V)$, where $\CC=\gamma(\cdot,0)$ is a proper immersion, so that $|\gamma(s,0)| \to \infty$ as $s\to \pm \infty$. Then $|\gamma(s,t)| \to \infty$ as $s\to \infty$ for all $t$, so that each map $\gamma(\cdot,t)$ is proper, and thus $\phi$ is proper.
\end{lem}
\begin{proof}
For each $t\in\RR$, since $|\gamma_t| \leq 1$, $|\gamma(s,t)|\geq |\gamma(s,0)| - \int_0^t |\gamma_t(s,\tilde{t})| d\tilde{t} \geq  |\gamma(s,0)| - t \to \infty$ as $s\to \pm \infty$.
\end{proof}
Recalling that $a_\pm(s) = v(s) \pm \dot{c}(s) = \gamma_t(s,0) \pm \gamma_s(s,0)$
give the spatial parts of the null vectors $A_\pm(s)=(1,a_\pm(s))$ along the initial tangent planes, with $|a_\pm(s)|^2=1$, from \eqref{dalembert2} we see
\eqq{\gamma_s(s,t) &= \frac{1}{2}\left( \dot{c}(s+t) + \dot{c}(s-t) + v(s+t) - v(s-t) \right) \\
&= \frac{1}{2} \left(a_+(s+t)-a_-(s-t)\right)}
so we have the following characterisation of $\Ksing$
\eq{\label{ksingcharacterisation}\Ksing = \left\{ (s,t) \in \RR^2 \colon a_+(s+t)=a_-(s-t)\right\}.} 
We will now show that $\Ssing$ is singular, at least in the sense that it consists of null points. The following result was observed, as part of a broader context, in \cite[Theorem 3.1]{jerrardetal}.
\begin{lem}\label{nulllemma}
Let $\phi\colon \RR^2 \to \EE^{1+2}$ be an evolution by isothermal gauge for a $C^1 \times C^0$ initial data $(\CC,V)$, and suppose $\Ksing$ as defined in \eqref{ksing} is non-empty. Suppose for some neighbourhood $U$ of a point $q\in \dd \Ksing$ that $\phi(U)$ is a $C^1$ embedded surface. Then $\phi(U)$ is null at $\phi(q)$. 
\end{lem} 
\begin{proof} Let $U$ be a neighbourhood of $q\in \dd \Ksing$ such that $\phi(U)$ is a $C^1$ embedded surface.
For each point $(s,t)\in U \setminus \Ksing$, the tangent space $T_{\phi(s,t)}\phi(U)$ is a timelike plane which intersects the light cone along null directions spanned by the nowhere vanishing null vectors 
$$\phi_s(s,t)+\phi_t(s,t)= \phi_s(s+t,0)+\phi_t(s+t,0)= A_+(s+t)=(1,a_+(s+t))$$ and 
$$\phi_s(s,t)-\phi_t(s,t)=\phi_s(s-t,0)-\phi_t(s-t,0) = A_-(s-t)=(1,a_-(s-t)).$$ 
Choose a sequence of points $(s_k,t_k)\in U \setminus \Ksing$ with $(s_k,t_k)\to q=(s_*,t_*)$. Since $a_+(s_*+t_*)=a_-(s_*-t_*)$, it follows that $\lim_{(s_k,t_k)\to(s_*,t_*)}a_+(s_k+t_k)=\lim_{(s_k,t_k)\to(s_*,t_*)}a_-(s_k-t_k)$, so the null lines along which $T_{\phi(s_k,t_k)}\phi(U)$ intersects the light cone converge. So $T_{\phi(s_*,t_*)}\phi(U)=\lim_{(s_k,t_k)\to(s_*,t_*)} T_{\phi(s_k,t_k)} \phi(U)$ must be a null plane. 
\end{proof}

\subsection{Some analysis of the singular set}\label{singularsetanalysis}
Let $\CC \colon \RR \to \RR^{1+2}$, be a $C^k$ immersion, $k\geq 1$, of the form $\CC(s)=(0,c(s))$, and let $V(s)=(1,v(s))$ be a $C^{k-1}$ timelike vector field along $\CC$ where $c,v$ satisfy \eqref{initialdataorthog1}, \eqref{initialdataorthog2}. Write
\eq{\label{initialunittangent}  U_0(s) = \frac{\dot{c}(s)}{|\dot{c}(s)|} }
for the unit tangent map along $\CC$. Let $\vartheta\colon \RR \to \RR$ be a lift of $U_0 \colon \RR \to S^1$, so that
\eq{\label{theta}U_0(s) = (\cos \vartheta(s), \sin \vartheta(s)) .}
If $\CC$ is $C^2$, then $\vartheta$ may be related to the curvature $k$ of $\CC$ by the formula
\eq{\label{curvatureidentity} \int_{s_1}^{s_2}k(s) d\sigma(s) = \vartheta(s_2)-\vartheta(s_1)}
where $d\sigma(s)=|\dot{c}(s)|ds$ is the element of arclength.

By \eqref{initialdataorthog1}, we may define a function $\mu\colon \RR \to (-1,1)$ such that 
\eq{\label{muidentity}v(s)= \mu(s)U_0(s)^\perp = \mu(s) (-\sin \vartheta(s), \cos \vartheta(s)).}
Next recall from \eqref{aplusminus'} that $a_\pm(s)= v(s) \pm \dot{c}(s)$. By trigonometric identities, it may be seen that the quantities
\eq{ \label{alphaplus} \alpha_+(s)&=\vartheta(s)+\arcsin(\mu(s)) \\
\label{alphaminus} \alpha_-(s) &=\vartheta(s)-\arcsin(\mu(s)) - \pi}
define a pair of lifts for $a_\pm$, so that
\al{a_\pm(s)=(\cos \alpha_\pm(s), \sin \alpha_\pm(s)).}
See Figure \ref{isothermalfigure}.
\begin{figure}[h]
\centering
\includegraphics[width=.3\textwidth]{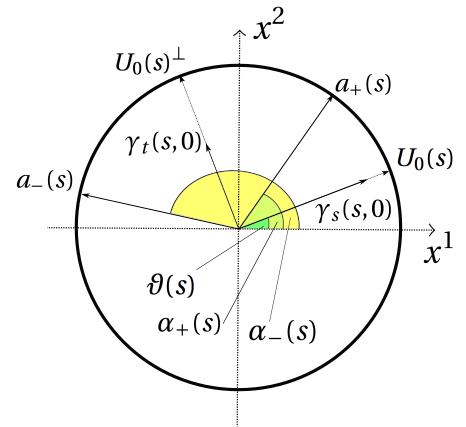}
\caption{The isothermal frame in angular coordinates.}\label{isothermalfigure}
\end{figure}

\begin{rem}
The function $\mu$ defined by \eqref{muidentity} may be given a geometric interpretation as follows. Defining $\varphi(s) = \arctanh \mu(s)$, we see that
\al{N(s) = (\sinh \varphi(s), - \cosh \varphi(s) \sin \vartheta(s), \cosh \varphi(s) \cos \vartheta(s))}
defines a spacelike unit normal to $T_{\CC(s)}\Sigma = \mathrm{span}\{ \dot{\CC}(s),V(s) \}$. So $(\vartheta, \varphi)$ are longitude-latitude coordinates on the 1-sheeted hyperboloid $S^{1+1}=\{ (x^0,x^1,x^2)\in \RR^{1+2} \colon -(x^0)^2+(x^1)^2+(x^2)^2 =1\}$. 
\end{rem}

Denote the characteristic diamond associated to the interval $[s_1,s_2]$ by
\eq{\label{characteristicdiamond} D(s_1,s_2)=\left\{ (s,t) \in \RR^2 \colon s_1+|t| \leq s \leq s_2-|t| \right\}.}

\begin{prop}\label{keyprop} Let $\phi\colon \RR^2 \to \EE^{1+2}$ be an evolution by isothermal gauge for a $C^1\times C^0$ initial data $(\CC,V)$. Writing $U_0$ for the unit tangent along $\CC$ as in \eqref{initialunittangent}, suppose that $\mathrm{Im}(U_0)$ contains a closed semi-circle, i.e.\ suppose there exist $s_1<s_2$ such that
\eq{\label{closedsemicircle} |\vartheta(s_2)-\vartheta(s_1)| \geq \pi,}
where $\vartheta$ is as in \eqref{theta}. Then, with $\Ksing$ as in \eqref{ksing} and $D(s_1,s_2)$ as in \eqref{characteristicdiamond} , it follows
\eq{\label{nonemptycharacteristictriangle} \Ksing \cap D(s_1,s_2)\neq \emptyset.}
\end{prop}
\begin{rem}
The same conclusion cannot be reached if $\mathrm{Im}(U_0)$ contains only a half-closed semi-circle. Indeed, in Example \ref{nongraphexample}, we had $\mathrm{Im}(\vartheta)=[-\frac{\pi}{2},\frac{\pi}{2})$, whilst $\Ksing = \emptyset$.
\end{rem}
\begin{proof}[Proof:]
Identities \eqref{alphaplus} and \eqref{alphaminus} give 
\eq{\label{volumeidentity} \left|\left(\alpha_+(s_2)-\alpha_+(s_1) \right) + \left( \alpha_-(s_2)-\alpha_-(s_1)\right) \right| = 2  \left| \vartheta(s_2) - \vartheta(s_1) \right| \geq 2\pi.}
It follows that $a_+([s_1,s_2])$ and $a_-([s_1,s_2])$ cannot form disjoint subsets of $S^1$. So there exist $\xi,\eta \in [s_1,s_2]$ such that $a_+(\xi)=a_-(\eta)$. Since $\Ksing$ is characterized by \eqref{ksingcharacterisation}, the proposition follows.
\end{proof}
We have the following consequence of Propositions \ref{keyprop} and \ref{regularityloss} (compare Theorem \ref{maintheorem}).
\begin{cor}\label{noC2cor} Let $\phi\colon \RR^2 \to \EE^{1+2}$ be an evolution by isothermal gauge for a $C^2\times C^1$ initial data $(\CC,V)$, and let $U_0$ be the unit tangent along $\CC$ as in \eqref{initialunittangent}. Suppose that $\mathrm{Im}(U_0)$ contains a closed semi-circle. Then $\Sigma=\Imm(\phi)$ is not a $C^2$ immersed surface.
\end{cor}
\begin{proof}
By Proposition \ref{keyprop}, 
\al{\Ksing \cap D(s_1,s_2) \neq \emptyset.} 
By continuity, for some $\eps>0$ we have
\al{\Ksing \cap D(s_1, s_2) \cap \{(s,t) \colon |t| <\eps \} =\emptyset,} 
so let $T\in\left(0,\frac{s_2-s_1}{2}\right]$ be the largest time such that $\Ksing \cap D(s_1, s_2) \cap \{(s,t) \colon |t| <T \} =\emptyset$ holds.
 
Supposing $(s_0,t_0)\in \Ksing \cap D(s_1, s_2) \cap \{(s,t) \colon t = T \}$, then the conditions for Proposition \ref{regularityloss} are satisfied on $\Omega= D(s_1, s_2) \cap \{(s,t) \colon |t| <T \}$, so the curvature $k(\cdot,t)$ of the cross sections $\gamma(\cdot,t)$ satisfies $\limsup_{t\nearrow t_0} k(s_0,t)=\infty$. For the case $(s_0,t_0)\in \Ksing \cap D(s_1, s_2) \cap \{(s,t) \colon t = -T \}$, a symmetric argument shows $\limsup_{t\searrow t_0} k(s_0,t)=\infty$.

If $\Sigma$ were $C^2$ immersed, then $\Sigma$ would be a causal surface, so $x^0|_\Sigma$ would have no critical points, and by the implicit function theorem the cross sections $\gamma(\cdot,t)$ would have locally uniformly bounded curvatures. This would amount to a contradiction, thus $\Sigma$ is not $C^2$. 
\end{proof}
In particular, Proposition \ref{keyprop} and Corollary \ref{noC2cor} apply to the case of a self-intersecting curve $\CC$, thanks to the following elementary result.
\begin{lem}\label{selfintersectionlemma}
Suppose $c \colon \RR \to \RR^2$ is a $C^1$ immersion with a point of self-intersection, i.e.\ there exist $r_1<r_2$ such that $c(r_1)=c(r_2)$. Let $U_0$ denote the unit tangent along $c$ as in \eqref{initialunittangent}. Then $U_0([r_1,r_2])$ contains an arc of length $>\pi$, i.e.\ there exist $s_1,s_2\in [r_1,r_2]$ such that 
\eq{\label{firstsharpinequality} \vartheta(s_2)-\vartheta(s_1) > \pi}
where $\vartheta$ denotes the angle swept out between $U_0$ and the $x^1$ axis as in \eqref{theta}.
\end{lem}
\begin{proof}
Since $c(r_1)=c(r_2)$, for every $\omega\in \RR^2$ we have
\al{\int_{r_1}^{r_2} \la \dot{c}(s), \omega \ra ds =0.}
But if $U_0([r_1,r_2])$ is contained in a closed semi-circle, then there exists $\omega_0\in \RR^2$ such that 
\al{\int_{r_1}^{r_2} \la \dot{c}(s), \omega_0 \ra ds >0,}
a contradiction. We conclude that $U_0([r_1,r_2])$ contains an arc of length $>\pi$ as claimed. 
\end{proof}
Proposition \ref{keyprop} gives a sufficient condition in terms of $\vartheta$ for $\Ksing \cap D(s_1,s_2)$ to be non-empty. We can also give a sufficient condition for no singularity in terms of $\vartheta$ and the initial velocity $v$.
\begin{lem}\label{nolocalsingularitylemma}
Let $\phi \colon \RR^2 \to \EE^{1+2}$ be an evolution by isothermal gauge for a $C^1\times C^0$ initial data $(\CC,V)$. Writing $\vartheta$ as in \eqref{theta} and $v$ as in \eqref{vequation} and \eqref{initialdataorthog1}, suppose that
\eq{\label{nolocalsingularityestimate} \sup_{r_1,r_2\in[s_1,s_2]} |\vartheta(r_2)-\vartheta(r_1)|^2 + \sup_{r\in[s_1,s_2]} |v(r)|^2 < 1.}
Then, with $\Ksing$ as in \eqref{ksing} and $D(s_1,s_2)$ as in \eqref{characteristicdiamond}, it follows
\eq{\label{nolocalsingularity}\Ksing \cap D(s_1,s_2) =\emptyset.}
\end{lem}
\begin{proof}
Writing $a_\pm$ as in \eqref{aplusminus}, it follows easily from \eqref{nolocalsingularityestimate} and trigonometric identities that $a_+(\xi)\neq a_-(\eta)$ for $\xi,\eta \in [s_1,s_2]$ (see Figure \ref{isothermalfigure}). As $\Ksing$ is characterised by \eqref{ksingcharacterisation}, the claim follows.
\end{proof}
\begin{rem}\label{globalexistenceremark}
If the initial data $(\CC,V)$ satisfies the estimate \eqref{nolocalsingularityestimate} on $[s_1, s_2]=\RR$, then by Lemma \ref{nolocalsingularitylemma} it follows $\Ksing = \emptyset$, so the evolution by isothermal gauge $\phi$ parameterises a properly immersed timelike maximal surface $\Imm(\phi)$ which contains $\CC$ and is tangent to $V$ along $\CC$. This is a global existence result which does not require any decay of initial data $(\CC,V)$ at infinity, and may be compared with recent results of \cite{LYY} and \cite{Wongsmalldata}. 
\end{rem}
\begin{cor}
Let $\CC\colon \RR \to \EE^{1+2}$ be given as $\CC(s)=(0,s,0)$ (i.e.\ $\mathrm{Im}(\CC)$ is a straight line) and let $V$ be any smooth timelike velocity along $\CC$. 
Let $\phi\colon \RR^2 \to \EE^{1+2}$ be the evolution of $(\CC,V)$ by isothermal gauge. Then $\Imm(\phi)$ is a smooth properly immersed timelike maximal surface containing $\mathrm{Im}(\CC)$ and tangent to $V$ along $\CC$.
\end{cor}
\begin{proof}
Since $\vartheta \equiv 0$ and $V$ is timelike, estimate \eqref{nolocalsingularity} holds on the interval $[s_1,s_2]=\RR$. So $\Ksing = \emptyset$ by Lemma \ref{nolocalsingularitylemma} and the claim follows.
\end{proof}
\begin{rem} If $\CC \colon \RR \to \{x^0=0\} \subseteq \EE^{1+2}$ is a smooth proper immersion such that $\Imm(\CC)$ is not a straight line, then it is easy to find a smooth vector field $V$ along $\CC$ for which the evolution $\phi$ of $(\CC,V)$ in isothermal gauge becomes singular in finite time (i.e. $\Ksing \neq \emptyset$). Indeed, let $\CC(s)=(0,c(s))$ and $U_0(s)=\nicefrac{\dot{c}(s)}{|\dot{c}(s)|}$, and choose $s_1,s_2\in \RR$ so that $U_0(s_1)\neq U_0(s_2)$. Let $\beta\in (0,2\pi)$ be such that $U_0(s_2)$ is given by an anti-clockwise rotation of $U_0(s_1)$ by $\beta$ degrees, and define $V(s)=(1,v(s))$ by $v(s)=\cos \frac{\beta}{2} U_0(s)^\perp$, where $\perp$ denotes an anti-clockwise rotation by $\frac{\pi}{2}$ degrees. Writing $a_\pm(s)=v(s)\pm \sin \frac{\beta}{2} U_0(s)$ for the spatial components of the null vectors $A_\pm(s)=(1,a_\pm(s))$ which span the tangent plane $T_{\CC(s)}\Imm(\phi) = \spann\{ \dot{\CC}(s), V(s) \}$, we may compute from the trigonometric identities \eqref{alphaplus} and \eqref{alphaminus} that $a_+(s_2)=a_-(s_1)$. So $\Ksing \neq \emptyset$ by \eqref{ksingcharacterisation}.
\end{rem}

From Lemma \ref{nolocalsingularitylemma} we may obtain the following short-time existence result, which does not require any decay of the initial data at infinity.
\begin{cor}[Short-time existence]\label{ste}
Let $\phi\colon \RR^2 \to \EE^{1+2}$ be an evolution by isothermal gauge for a $C^k\times C^{k-1}$ initial data $(\CC,V)$, $k\geq1$, and let $U_0$ denote the unit tangent vector along $\CC$ as in \eqref{initialunittangent}. Suppose $V$ is uniformly timelike, i.e.\ with $V=(1,v)$ we have $\sup_{s\in\RR} |v(s)| < 1$, and suppose $U_0$ is uniformly continuous. Then there exists $T>0$ depending only on $\sup_{s\in\RR}\frac{1}{1-|v(s)|}$ and the modulus of continuity of $U_0$ such that $\Imm(\phi)\cap \{(x^0,x^1,x^2)\colon |x^0|\leq T\}$ is a $C^k$ immersed timelike maximal surface containing $\mathrm{Im}(\CC)$ and tangent to $V$ along $\CC$. 
\end{cor}
\begin{proof}
Take $\eps>0$ so that $\sup_{s\in\RR}|v(s)|^2\leq 1-\eps$. Since $U_0$ is uniformly continuous, there exists $\delta>0$, depending on $\eps$ and the modulus of continuity of $U_0$, such that $|\vartheta(r_2)-\vartheta(r_1)|^2<\eps$ provided $|r_1-r_2|\leq\delta$. Defining $s_k=\frac{\delta k}{2}$ for all $k\in \ZZ$ gives
\al{ \sup_{r_1,r_2\in [s_k,s_{k+2}]} |\vartheta(r_2)-\vartheta(r_1)|^2  + \sup_{r\in [s_k,s_{k+2}]} |v(r)|^2 < 1,}
so $\Ksing \cap D(s_k,s_{k+2})=\emptyset$ for all $k\in \ZZ$ by Lemma \ref{nolocalsingularitylemma}. With $T=\frac{\delta}{4}$, the set $\{(s,t) \in \RR^2 \colon |t|\leq T\}$ is contained in $\cup_{k\in\ZZ} D(s_k,s_{k+2})$, so $\Ksing \cap \{(s,t)\in \RR^2 \colon |t|\leq T\} = \emptyset$, and the claim follows.
\end{proof}
\subsection{Examples of $C^1$ properly embedded surfaces which are smooth timelike maximal surfaces away from some null set}\label{c1extendibleexamples}
We will now give some (non-generic) examples where the Cauchy evolution for a timelike maximal surface becomes singular in finite time, but the evolution in isothermal gauge beyond singular time yields a $C^1$ embedded surface.
\begin{ex}[$C^1$ embedded maximal surfaces which are smooth away from a pair of null lines and contain non-graphical compact sets]\label{regularexample1} Let $l_1$ and $l_2$ be the parallel half lines which take their endpoints at $(-\frac{1}{2},0)$ and $(\frac{1}{2},0)$ and which are obtained as left and right translations respectively by a distance $\frac{1}{2}$ of the upper $x^2$-axis. Let $f$ be a smooth segment of embedded curve of length $2L>0$, which smoothly joins $l_1$ and $l_2$ at their endpoints, such that the unit tangent along $f$ has non-vanishing $x^1$ component everywhere except at the endpoints. See Figure \ref{singulargrimreaperevolution}(a). Let $c\colon \RR \to \RR^2$ be a parameterisation of $l_1\cup l_2 \cup f$ by arclength,  
\al{c(s)=\begin{cases} 
\left(-\frac{1}{2}, -s -L\right) &\text{for} \hskip3pt s\in(-\infty,-L] \\
\left( f^1(s), f^2(s) \right) &\text{for} \hskip3pt s\in(-L,L) \\
\left(\frac{1}{2}, s-L\right) &\text{for} \hskip3pt s\in[L,\infty) .
\end{cases}}
Writing $\dot{c}(s)=(\cos \vartheta(s), \sin \vartheta(s))$, and we see $\mathrm{Im}(\vartheta)=[-\frac{\pi}{2},\frac{\pi}{2}]$. Moreover, $\vartheta(s)\in (-\frac{\pi}{2},\frac{\pi}{2})$ for $s\in (-L,L)$. 

The evolution of $\CC(s)=(0,c(s))$ with initial velocity $V=(1,0,0)$ by isothermal gauge is $\phi(s,t)=(t,\gamma(s,t))$, where $\gamma(s,t)=\frac{1}{2}\left(c(s+t)+c(s-t)\right)$. By Proposition \ref{keyprop}, it follows that $\Ksing$, as defined in \eqref{ksing}, is non-empty. We will now compute $\Ksing$ explicitly. Since $\dot{c}(s+t)=-\dot{c}(s-t)$ if and only if $s-t\leq-L$ whilst $s+t\geq L$ or $s-t\geq L$ whilst $s+t\leq-L$, it follows $\Ksing=\Ksing^+\cup\Ksing^-$ where
\al{\Ksing^+&=\left\{(s,t) \colon t\geq L,\hskip3pt L-t \leq s \leq t-L \right\}, \\
 \Ksing^- &=\left\{(s,t) \colon t\leq-L,\hskip3pt t+L \leq s \leq -t-L \right\}.}
We then see $\Ssing = \Ssing^+ \cup \Ssing^-$, where
\al{\Ssing^+ &= \left\{ (t,0,t-L) \colon t\geq L \right\}, \\
 \Ssing^- &= \left\{ (t,0,-t-L) \colon t\leq -L \right\}.}
i.e.\ $\Ssing$ consists of a pair of null half-lines, one emanating towards the future from the point $(L,0,0)$ and one emanating towards the past from the point $(-L,0,0)$. $\Sigma\setminus \Ssing$ is a smooth immersed timelike maximal surface.

Note that the unit tangent $\dot{c}(s)$ is always confined to a closed semi-circle as $\dot{c}^1(s)\geq 0$. Writing $U(s,t)=\frac{\gamma_s(s,t)}{|\gamma_s(s,t)|}=\frac{\dot{c}(s+t)+\dot{c}(s-t)}{|\dot{c}(s+t)+\dot{c}(s-t)|}$ for the spatial unit tangent, defined {a priori} for $(s,t)\in\RR^2\setminus\Ksing$, it is seen that $\lim_{(s,t)\to\Ksing}U(s,t)=(1,0)$. Thus $U(s,t)$ extends continuously to a unit tangent vector field along $\gamma(s,t)$. It may then be seen to follow that $\Sigma$ is a $C^1$ immersed causal surface. See Figure \ref{singulargrimreaperevolution}(b).

Applying Proposition \ref{regularityloss}, we see that the curvature of the cross sections $\gamma(\cdot,t)$ blows up as $t\nearrow L$, so $\Sigma$ is not a $C^2$-immersed surface in any neighbourhood of $\phi(0,L)$. Since $\gamma(s,t)=\gamma(s,-t)$, we see that $\Sigma$ is invariant under a reflection through the $\{x^0 = 0 \}$ plane, and so $\Sigma$ is not a $C^2$ immersed surface in any neighbourhood of $\phi(0,-L)$. 

It is easy to find a compact subset $K\subseteq \Sigma$ which is not a graph. We observe that the image of the spacelike unit normal in this example (defined on $\Sigma \setminus \Ssing$) is contained in a closed hemi-hyperboloid.
\begin{figure}[h]
    \centering
    	\captionsetup{width=.9\linewidth}
    \begin{tabular}{cc}
    \begin{minipage}{0.35\textwidth}
        \centering
        \includegraphics[width=\textwidth]{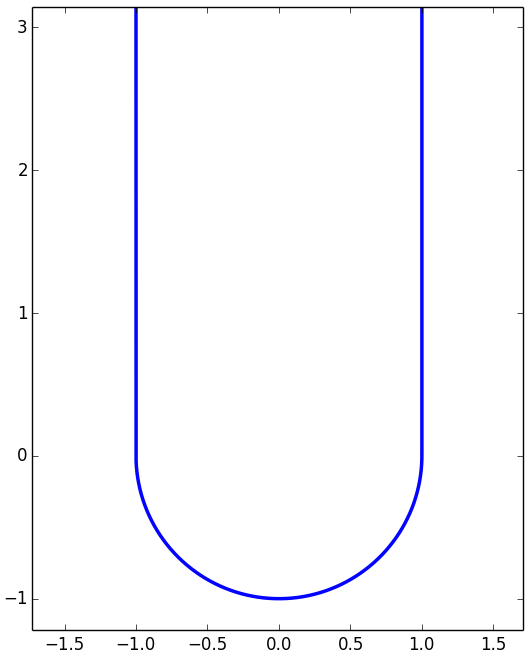}
            \end{minipage}\hfill
    &
    \begin{minipage}{0.55\textwidth}
        \centering
        \includegraphics[width=\textwidth]{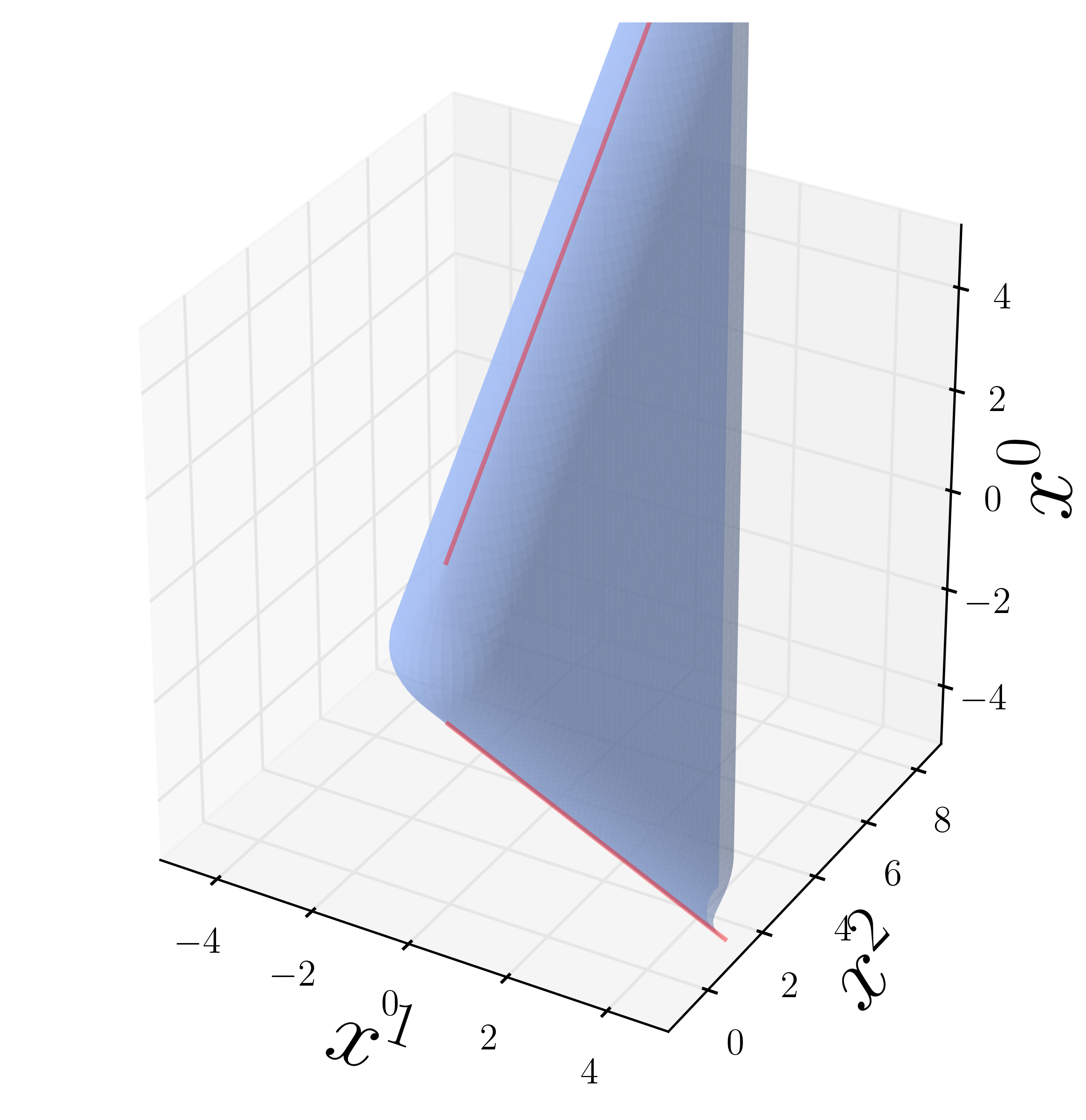}
    \end{minipage}\\
    (a) & (b)
    \end{tabular}
    \caption[$C^1$ surfaces which are timelike maximal surfaces away from a pair of null half lines \& contain non-graphical compact subsets]{  (a) A cigar curve which contains a compact subset which is not a graph. (b) Evolution of (a) by isothermal gauge to a $C^1$ embedded maximal surface $\Sigma$ which is null along null lines $\Ssing$ shown in red. There is a compact subset $K\subseteq \Sigma$ which is not a graph.
      }
    \label{singulargrimreaperevolution}
\end{figure}
\end{ex}
\begin{ex}[$C^1$ embedded doubly-periodic maximal surfaces which are smooth away from isolated null points situated on a rectangular lattice, and which are graphs, but not $C^1$ graphs]\label{regularexample2}
Let $f=(f^1,f^2)\colon[0,L] \to \RR^2$, parametrize a section of curve by arclength, so that $f(0)=(-1,0)$, $f(L)=(1,0)$, $\dot{f}^1(s)>0$ for $s\in (0,L)$, $\dot{f}^1(0)=\dot{f}^1(L)=0$ and $\frac{d^k f^2}{ds^k}(0)=\frac{d^k f^2}{ds^k}(L)=0$ for $k\geq 2$. Now extend $f$ periodically to a smooth immersion $c\colon \RR \to \RR^2$ by
\al{c(s)=\begin{cases} 
\left(f^1(s),f^2(s)\right) &\text{for} \hskip3pt s\in[0,L] \\
\left(2+f^1(s),-f^2(s) \right) &\text{for} \hskip3pt s\in(L,2L) \\
\left(4n,0\right)+c(s-2nL) &\text{for} \hskip3pt s\in[2nL,2(n+1)L), n\in \ZZ \setminus \{0\} .
\end{cases}}
See Figure \ref{periodicsurface}(a). It may be seen that $\mathrm{Im}(c)$ defines a graph over the $x^1$ axis, but not a $C^1$ graph. 

As $c$ is parametrized by arclength, the evolution by isothermal gauge $\phi(s,t)=(t,\gamma(s,t))$ of the curve $\CC(s)=(0,c(s))$ with initial velocity $V=(1,0,0)$ is given by $\gamma(s,t)=\frac{1}{2} \left( c(s+t) + c(s-t) \right)$. Note that $(s,t)\in \Ksing$ if $\frac{s+t}{L}$ is an odd integer and $\frac{s-t}{L}$ is an even integer or vise-versa. From this we deduce that
\al{\Ksing= \left\{\left(\frac{mL}{2},\frac{nL}{2}\right) \colon m \hskip2pt \text{and} \hskip2pt n \hskip2pt \text{are odd integers} \right\}}
and since $c(\frac{nL}{2})=(n-1,0)$ for all $n\in \ZZ$, we have
\al{\Ssing= \left\{\left(\frac{nL}{2}, k, 0 \right) \colon n \hskip3pt \text{is an odd integer and} \hskip3pt k \hskip 3pt \text{is an even integer} \right\} }
which is a rectangular lattice of isolated points.

$\Sigma$ is a smooth, timelike immersed surface away from $\Ssing$, and again we observe that $\dot{c}^{1}(s)\geq0$, and so $\lim_{(s,t)\to\Ksing}U(s,t)=(1,0)$, and thus $\Sigma$ is a $C^1$ immersed surface. By Proposition \ref{regularityloss} we see that $\Sigma$ is not a $C^2$ immersed surface in any neighbourhood of a point in $\Ssing$. $\Sigma$ is a graph over the $x^0$--$x^1$ plane, but not a $C^1$ graph. See Figure \ref{periodicsurface}(b). 
\begin{figure}[h]
    \centering
    	 \captionsetup{width=.9\textwidth}
    \begin{tabular}{cc}
    \begin{minipage}{0.35\textwidth}
        \centering
        \includegraphics[width=\textwidth]{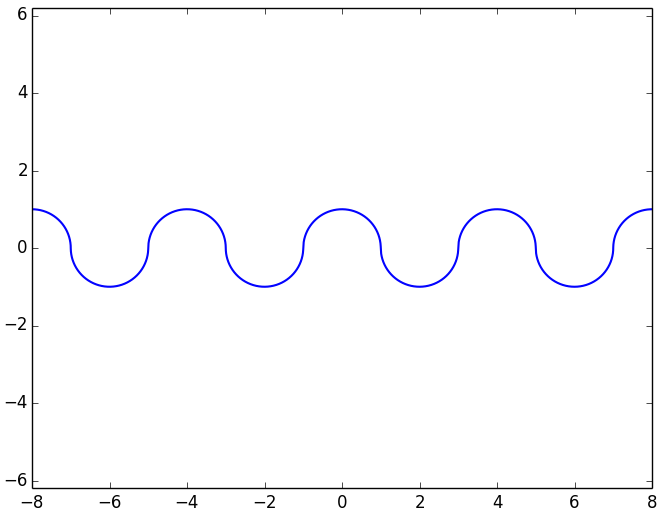}
            \end{minipage}\hfill
    &
    \begin{minipage}{0.55\textwidth}
        \centering
        \includegraphics[width=\textwidth]{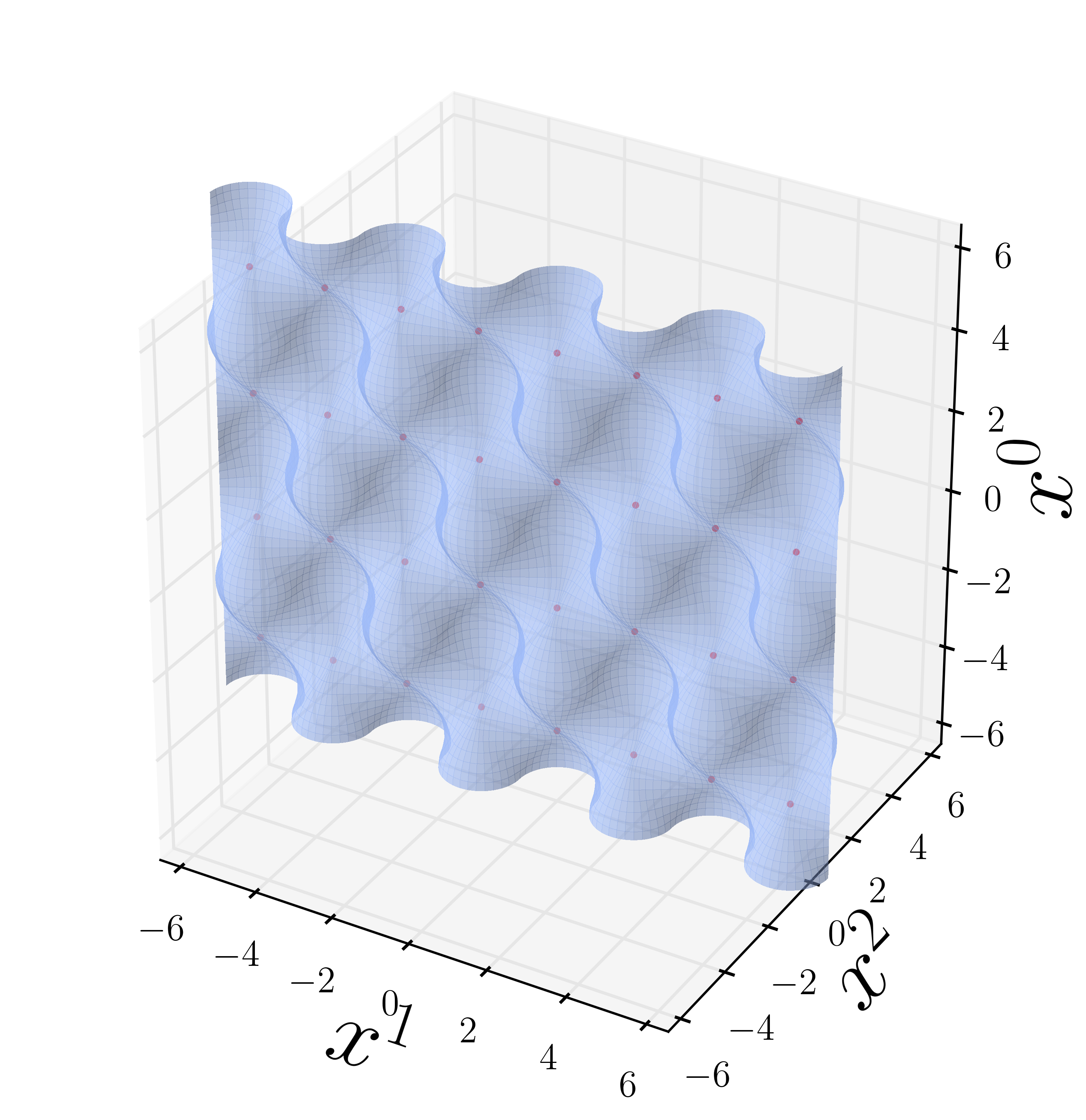}
    \end{minipage}\\
    (a) & (b)
    \end{tabular}
    \caption[$C^1$ doubly-periodic surfaces which are timelike maximal surfaces away from a lattice of isolated null points \& contain non-$C^1$-graphical compact subsets]{  (a) A periodic wedge of hemi-circles which is a graph, but not a $C^1$ graph. (b) Evolution of (a) to a $C^1$ doubly-periodic maximal surface with null points $\Ssing$ on a rectangular lattice shown in red. $\Sigma$ is a graph over the $x^0$--$x^1$ plane, but not a $C^1$ graph.
      }
\label{periodicsurface}
\end{figure}
\end{ex}
\subsection{Discontinuity of the spatial unit tangent: proof of Theorem \ref{gaugetheorem}}\label{gaugesection}
The surfaces constructed by isothermal gauge in Example \ref{regularexample1} are $C^1$ embedded, are smooth timelike maximal surfaces away from a pair of null lines, and contain compact subsets which are non-graphical (compare with Theorem \ref{maintheorem}).
Note that in Examples \ref{regularexample1} and \ref{regularexample2}, the image of the tangent vector $U_0$ along the initial curve $\CC$ is exactly a closed semi-circle.

In this section we will show that the behaviour observed in Example \ref{regularexample1} is borderline. To be precise, we will prove Theorem \ref{gaugetheorem} which states that: if $\phi\colon \RR^2 \to \EE^{1+2}$ is an evolution by isothermal gauge for a $C^1\times C^0$ initial data $(\CC,V)$, and if the image of the unit tangent vector along $\CC$ contains an arc of length $>\pi$, i.e.\ if there exist $s_1, s_2\in \RR$ so that
\eq{\label{sharpcondition}\vartheta(s_2)-\vartheta(s_1) > \pi,}
where $\vartheta$ is as in \eqref{theta}, then the spatial unit tangent (defined along $\phi|_{\RR^2\setminus \Ksing}$) admits no extension to a continuous unit tangent vector field along $\phi$. 

When $\CC$ is a closed curve, the discontinuity of the spatial unit tangent was proved by Nguyen \& Tian \cite[Prop.\ 2.9 \& Prop.\ 2.11]{NT} (for smooth initial data) and by Jerrard, Novaga and Orlandi in \cite[Theorem 5.1]{jerrardetal} (for $C^1\times C^0$ initial data). The proof of Theorem \ref{gaugetheorem} extends the argument of those authors to the spatially non-compact case (note that if $\CC$ is closed, then there exist $s_1,s_2$ so that \eqref{sharpcondition} is satified by Lemma \ref{selfintersectionlemma}).

Let $\phi\colon \RR^2 \to \EE^{1+2}$, $\phi(s,t)=(t,\gamma(s,t))$ be an evolution by isothermal gauge. As in Section \ref{singularsetanalysis}, we write 
\al{a_\pm(s)=\gamma_t(s,0) \pm \gamma_s(s,0),} so that $|a_\pm(s)|^2=1$. Recall from \eqref{alphaplus}, \eqref{alphaminus} that $a_\pm(s)=(\cos \alpha_\pm (s) , \sin \alpha_\pm(s))$, where
\al{\alpha_+(s) &= \vartheta(s) + \arcsin \left(\mu(s) \right) \\
\alpha_-(s) &= \vartheta(s) - \arcsin \left( \mu(s) \right) - \pi,} 
where $\vartheta$ and $\mu$ are defined by \eqref{theta} and \eqref{muidentity}.

Let us now introduce 
\eq{\label{beta}\beta(s,t)=\alpha_+(s+t)-\alpha_-(s-t).}
We have
\eq{\label{initialbeta}\beta(s,0)=\alpha_+(s)-\alpha_-(s)=2\arcsin(\tanh(\varphi(s)))+\pi \in (0,2\pi) \quad \text{for all}\hskip2pt s\in\RR.}
The proof of Theorem \ref{gaugetheorem} is via a study of the spatial unit tangent map 
\al{U(s,t)= \frac{\gamma_s(s,t)}{|\gamma_s(s,t)|},} 
which is well defined for $(s,t)\in \RR^2 \setminus \Ksing$. From \eqref{dalembert2} one may compute explicitly
\eqq{\label{tangentidentity} U(s,t) = \sgn\left(\sin \frac{\beta(s,t)}{2}\right) e(s,t) }
where 
\eq{\label{eidentity} e(s,t)=\left(-\sin \frac{\alpha_+(s+t)+\alpha_-(s-t)}{2}, \cos \frac{\alpha_+(s+t)+\alpha_-(s-t)}{2}\right)} 
is a continuous unit vector field along $\gamma(s,t)$ (note that $e(s,t)$ does not necessarily define a unit tangent vector field along $\gamma(s,t)$). 

We have $(s,t)\in \Ksing$ precisely when $\beta(s,t)\in 2\pi \ZZ$. From formula \eqref{tangentidentity}, it is apparent that to study when $U$ becomes discontinuous requires analysis of when $\sin\left(\frac{\beta}{2}\right)$ changes sign. 
\begin{lem}\label{betaswitcherlemma} Let $\phi\colon \RR^2 \to \EE^{1+2}$ be an evolution by isothermal gauge for a $C^1\times C^0$ initial data $(\CC,V)$. With $U_0$ denoting the unit tangent along $\CC$ as in \eqref{initialunittangent}, suppose that $\Imm(U_0)$ contains an arc of length $>\pi$, i.e.\ suppose there exist $s_1,s_2\in \RR$ such that
\al{\vartheta(s_2)-\vartheta(s_1) > \pi,}
where $\vartheta$ is as in \eqref{theta}. Then, with $\beta$ as in \eqref{beta}, there exists $(s_0,t_0)$ such that $\beta(s_0,t_0) \notin [0,2 \pi]$. Furthermore, if $\CC\colon \RR \to \{x^0=0\} \subseteq \RR^{1+2}$ is a proper immersion, then there exists a time $t_*\in \RR$ such that $\sin\left(\frac{\beta(\cdot,t_*)}{2}\right)$ takes both positive and negative values. 
\end{lem}
\begin{proof}
By identities \eqref{alphaplus} and \eqref{alphaminus}, we have 
\al{\big(\alpha_+(s_2)-\alpha_-(s_1)\big)-\big(\alpha_+(s_1)-\alpha_-(s_2)\big) = 2\big(\vartheta(s_2)-\vartheta(s_1)\big) > 2\pi}
and so, setting $s_0=\frac{1}{2}(s_1+s_2)$ and $t_0=\frac{1}{2}(s_1-s_2)$ gives
\al{\beta(s_0,-t_0) - \beta(s_0,t_0) &= \big(\alpha_+(s_0-t_0) - \alpha_-(s_0+t_0)\big) \\
&\hskip25pt -\big(\alpha_+(s_0+t_0)-\alpha_-(s_0-t_0)\big) \\
&> 2\pi.}
It follows that one of $\beta(s_0,-t_0)>2\pi$ or $\beta(s_0,t_0)<0$ holds, so $\beta(s_0,t_0)\notin [0,2\pi]$ for some $(s_0,t_0)$ as claimed.

Now suppose $\CC$ is proper, and suppose for a contradiction that there exists no time $t_*$ such that $\sin\left(\frac{\beta(\cdot,t_*)}{2}\right)$ takes both positive and negative values. Write $A=\{ t\in \RR \colon \sin\frac{\beta(s,t)}{2}\geq 0 \hskip2pt\text{for all}\hskip2pt s\in \RR\}$, $B=\{ t\in \RR \colon \sin\frac{\beta(s,t)}{2}\leq 0 \hskip2pt\text{for all}\hskip2pt s\in \RR\}$. Then $A$ and $B$ are closed sets, and we are supposing that $A \cup B = \RR$.

Note that $A$ is non-empty by \eqref{initialbeta}, whilst $\beta(s_0,t_0)\notin[0,2\pi]$ implies that $B$ is non-empty, and so by connectedness of $\RR$, $A\cap B$ must be non-empty. Taking $t_1\in A\cap B$ gives $\beta(\cdot,t_1)\equiv 2 k\pi$, which implies $\gamma_s(\cdot,t_1)\equiv 0$ so $\mathrm{Im}(\gamma(\cdot,t_1))$ consists of a single point. But since $\CC$ is proper, this contradicts Lemma \ref{propernesslemma}. Thus the lemma is proved. 
\end{proof}
We will deduce Theorem \ref{gaugetheorem} from Lemma \ref{betaswitcherlemma} together with the following 
\begin{lem}\label{jerrardslemma}
Let $\phi\colon \RR^2 \to \EE^{1+2}$, $\phi(s,t)=(t,\gamma(s,t))$ be an evolution by isothermal gauge for a $C^1\times C^0$ initial data $(\CC,V)$, and let $\beta$ be as in \eqref{beta}. Suppose there exists $t_*\in \RR$ such that $\sin\left(\frac{\beta(\cdot,t_*)}{2}\right)$ takes both positive and negative values on an interval $[s_1,s_2]$. Then for any $\zeta>0$, there is an open interval $I\subseteq (t_*-\zeta, t_*+\zeta)$, such that for all $t\in I$, either $\gamma([s_1,s_2],t)$ is not a $C^1$ immersed curve, or $\gamma([s_1,s_2],t)$ is a $C^1$ immersed curve but $U(\cdot,t)=\nicefrac{\gamma_s(\cdot,t)}{|\gamma_s(\cdot,t)|}$ admits no continuous extension to a unit tangent vector field along $\gamma(\cdot,t)$ on $[s_1,s_2]$.
\end{lem}
\begin{proof} We will follow the proof of \cite[Theorem 5.1(iii)]{jerrardetal}. Let us assume that $[s_1,s_2]$ is such that $\beta(s_1,t_*)<0$ and $\beta(s_2,t_*)>0$, since all other cases may be treated similarly. By continuity there exists $\delta_0\in(0,\zeta]$ such that $\beta(s_1,t)<0$ and $\beta(s_2,t)>0$ for all $t\in(t_*-\delta_0,t_*+\delta_0)$.

Suppose for some $t_0\in(t_*-\delta_0,t_*+\delta_0)$, we have that $\gamma([s_1,s_2],t_0)$ is a $C^1$ immersed curve and $U(\cdot,t_0)$ extends to a continuous unit vector field $\hat{U}(\cdot,t_0)$ along $\gamma(\cdot,t_0)$ on the interval $[s_1,s_2]$ (we will see such a situation in Example \ref{cuspswitcherexample}). Define
\eqq{\label{r1r2}r_2 &= \sup \{ \hat{s}\in[s_1,s_2] \colon \beta(s,t_0) \leq 0 \quad \text{for all} \hskip2pt s\in[s_1,\hat{s}] \} \\
r_1 &= \inf \{ \hat{s}\in[s_1,r_2] \colon \beta(s,t_0) = 0 \quad \text{for all} \hskip2pt s\in[\hat{s},r_2] \},}
then
\eq{\label{betazero}\beta(s,t_0)=0 \quad \text{for all} \hskip2pt s\in[r_1,r_2]} 
and $\beta$ takes both positive and negative values in every neighbourhood of $[r_1,r_2]$. 
\begin{figure}[h]
\centering
\includegraphics[width=.7\textwidth]{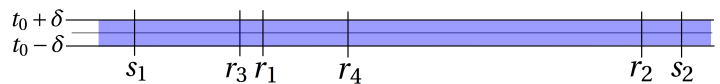}
\caption{The terms in the proof of Lemma \ref{jerrardslemma}.}\label{analysis_domain}
\end{figure}

We claim that 
\eq{\label{curveregularitycriterion}\alpha_+(r_1+t_0)=\alpha_+(r_2+t_0) + m \pi \quad \text{for some odd integer} \hskip2pt m.} 
To show \eqref{curveregularitycriterion}, note that since $\gamma(s,t_0)=\gamma(r_1,t_0)$ for all $s\in[r_1,r_2]$, it follows that $\hat{U}(r_1,t_0)=\hat{U}(r_2,t_0)$. Take sequences $\{x_n\}$ and $\{y_n\}$ with $x_n \to r_1$, $\beta(x_n,t_0)<0$, and $y_n \to r_2$,  $\beta(y_n,t_0)>0$ (which is possible from the definitions of $r_1$ and $r_2$). Then from \eqref{tangentidentity}
\al{\hat{U}(r_1,t_0) &= \lim\limits_{x_n\to r_1} \left\{ \sgn \left( \sin \frac{\beta(x_n,t_0)}{2} \right)  \right\} e(r_1,t_0) \\
 \hat{U}(r_2,t_0) &= \lim\limits_{y_n\to r_2} \left\{ \sgn \left( \sin \frac{\beta(y_n,t_0)}{2} \right)  \right\} e(r_2,t_0) \\ 
 &= -\lim\limits_{x_n\to r_1} \left\{ \sgn \left( \sin \frac{\beta(x_n,t_0)}{2} \right)  \right\} e(r_2,t_0),}
so $e(r_1,t_*)=-e(r_2,t_*)$ from which \eqref{curveregularitycriterion} follows from \eqref{eidentity} and \eqref{betazero}.

Geometrically, \eqref{betazero} and \eqref{curveregularitycriterion} amount to the statement that $\alpha_+(s+t_0)$ and $\alpha_-(s-t_0)$ (which we recall represent the null directions along the initial curve) are identically equal for $s\in[r_1,r_2]$, and undergo a rotation by a non-trivial multiple of $\pi$ as $s$ varies from $r_1$ to $r_2$. We will now show that this situation will be lost after a small perturbation of $t_0$. More precisely, we will show that for any $\eps>0$, there is an open interval $I$, either of the form $I=(t_0,t_0+\delta)$ or $I=(t_0-\delta,t_0)$ for some $\delta>0$, such that for each $t\in I$, there is an interval $J=J(t)\subseteq[s_1,s_2]$ such that $\beta(\cdot,t)$ takes both positive and negative values on $J$ and $|\alpha_+(w_1+t)-\alpha_+(w_2+t)|<\eps$ for all $w_1,w_2\in J$. Taking $\eps$ smaller than $\pi$, this will imply that condition \eqref{curveregularitycriterion} with $t_0$ replaced by $t$ cannot hold for any $r_1,r_2\in J$, so we will conclude that for each $t\in I$, the unit tangent $U(\cdot,t)$ admits no continuous extension to a unit tangent map, from which the conclusion of the lemma will follow.

Fix $\eps>0$. By \eqref{r1r2} and continuity of $\alpha_+$ there exists $r_3\in [s_1,r_1)$ such that $\beta(r_3,t_0)<0$ and 
\eq{\label{smallspacestep}|\alpha_+(s+t_0)-\alpha_+(r_1+t_0)|<\frac{\eps}{4} \quad \text{for}\hskip3pt s\in[r_3,r_1].}
Take $\delta>0$ so that 
\eq{\label{lefthandbelow} \beta(r_3,t) &< 0 \quad \text{for}\hskip3pt t\in [t_0-\delta,t_0+\delta].}
By the uniform continuity of $\alpha_+$ on compact sets, by refining $\delta>0$ to a smaller number as necessary, we may ensure
\eq{
\label{smalltimestep} |\alpha_+(s+t) - \alpha_+(s+t_0)|&< \frac{\eps}{4} \quad \text{for}\hskip3pt s\in[s_1,s_2], t\in[t_0-\delta,t_0+\delta].}
By \eqref{curveregularitycriterion}, we can define
\eq{\label{r3identity}r_4 = \inf \left\{ s\in[r_1,r_2] \colon |\alpha_+(s+t_0)-\alpha_+(r_1+t_0)| = \frac{\eps}{4}\right\}.}

We will first treat the case where $\alpha_+(r_4+t_0)=\alpha_+(r_1+t_0)+\frac{\eps}{4}$. By refining $\delta>0$ to be smaller as neccesary, we may assume that $\alpha_+(w_2+t_0)>\alpha_+(w_1+t_0)$ provided $w_1\in[r_1,r_1+\delta]$ and $w_2\in[r_4-\delta,r_4]$. Then for each $\tau\in(0,\delta]$, we have
\al{\int_{r_1}^{r_4-\tau} (\alpha_+(s+\tau+t_0)-\alpha_+(s+t_0)) ds &= \int_{r_4-\tau}^{r_4} \alpha_+(s+t_0) ds \\
 & \hskip25pt - \int^{r_1+\tau}_{r_1} \alpha_+(s+t_0) ds \\
&>0}
which shows that there exists an $s(\tau)\in[r_1,r_4-\tau]$ such that $\alpha_+(s(\tau)+\tau+t_0)>\alpha_+(s(\tau)+t_0)$. We then see 
\eqq{\label{righthandabove}\beta\left(s(\tau)+\frac{\tau}{2},t_0+\frac{\tau}{2}\right)&=\alpha_+(s(\tau)+\tau+t_0)-\alpha_+(s(\tau)-t_0) \\
&>\alpha_+(s(\tau)+t_0)-\alpha_+(s(\tau)-t_0)=\beta(s(\tau),t_0) \\
&\stackrel{\eqref{betazero}}{=}0.}
Then for all $\tau\in(0,\delta]$, by \eqref{lefthandbelow} and \eqref{righthandabove} $\beta(\cdot,t_0+\frac{\tau}{2})$ takes both positive and negative values on $J=J(t_0+\frac{\tau}{2}):=[r_3,s(\tau)+\frac{\tau}{2}]$. On the other hand, for all $\omega_1,\omega_2\in J$ we have
\al{|\alpha_+(\omega_1+t_0+\frac{\tau}{2})-\alpha_+(\omega_2+t_0+\frac{\tau}{2})|&\leq |\alpha_+(\omega_1+t_0+\frac{\tau}{2})-\alpha_+(\omega_1+t_0)| \\ + |\alpha_+(\omega_2+t_0+\frac{\tau}{2})-\alpha_+(\omega_2+&t_0)| +|\alpha_+(\omega_1+t_0)-\alpha_+(r_1+t_0)| \\ &+ |\alpha_+(\omega_2+t_0)-\alpha_+(r_1+t_0)|}
and since the first two terms on the right hand side of the above inequality are bounded by  \eqref{smalltimestep} and each of the last two terms is bounded by \eqref{smallspacestep} and \eqref{r3identity}, this gives $|\alpha_+(\omega_1+t_0+\frac{\tau}{2})-\alpha_+(\omega_2+t_0+\frac{\tau}{2})|<\eps$ which is what we set out to show.

Next we treat the case where $\alpha_+(r_4+t_0)=\alpha_+(r_1+t_0)-\frac{\eps}{4}$ by a similar argument. Choose $\delta>0$ so that $\alpha_+(w_1+t_0)>\alpha_+(w_2+t_0)$ provided $w_1\in[r_1,r_1+\delta]$ and $w_2\in[r_4-\delta,r_4]$. Then, for all $\tau\in(0,\delta]$ there exists $s(\tau)\in[r_1+\tau,r_4]$ such that $\alpha_+(s(\tau)-\tau+t_0)>\alpha_+(s(\tau)+t_0)$. In this case,
\eqq{\label{righthandabove2}\beta\left(s(\tau)-\frac{\tau}{2},t_0-\frac{\tau}{2}\right) &= \alpha_+(s(\tau)-\tau+t_0) - \alpha_-(s(\tau)-t_0) \\
 &> \alpha_+(s(\tau)+t_0)-\alpha_-(s(\tau)-t_0) = \beta(s(\tau),t_0) \\
 &\stackrel{\eqref{betazero}}{=} 0 .}

Then for all $\tau \in (0, \delta]$, by \eqref{lefthandbelow} and \eqref{righthandabove2} $\beta(\cdot, t_0-\frac{\tau}{2})$ takes both positive and negative values on $J=[r_3, s(\tau)-\frac{\tau}{2}]$, whilst for all $w_1,w_2 \in J$, arguing as above by \eqref{smallspacestep}, \eqref{smalltimestep} and \eqref{r3identity} we have $|\alpha_+(w_1+t_0-\frac{\tau}{2})-\alpha_+(w_2+t_0-\frac{\tau}{2})|<\eps$ which is what we set out to show. This completes the proof.
\end{proof}
The interval $I$ in Lemma \ref{jerrardslemma} may be chosen to be contained in any neighbourhood of the time $t_*$, but it is not always possible to choose an interval $I$ containing $t_*$. Indeed, it is possible that $\sin\left(\frac{\beta(\cdot,t_*)}{2}\right)$ takes both positive and negative values on an interval $[s_1,s_2]$ whilst $\gamma([s_1,s_2],t_*)$ is a $C^1$ immersed curve and $U(\cdot,t_*)$ admits a continuous extension to a unit tangent vector field along $\gamma(\cdot,t_*)$, as the following example illustrates.
\begin{ex}[Cusp reversal]\label{cuspswitcherexample}
Consider the $C^1$ initial curve defined by 
\al{c(s)= \begin{cases}
\left(s,-1\right) &\text{for} \hskip3pt s\in(-\infty,0] \\
\left(\sin s, -\cos s\right) &\text{for} \hskip3pt s\in (0,2\pi] \\
\left( \frac{1}{2} \sin 2s , \frac{-1}{2}(1+\cos 2s) \right) &\text{for} \hskip3pt s\in (2\pi, \frac{9\pi}{4}] \\
\left( \frac{1}{2}, -\frac{1}{2} + s - \frac{\pi}{4} \right) &\text{for} \hskip3pt s\in (\frac{9\pi}{4},\infty).
\end{cases}}
See Figure \ref{cuspswitcher}(a). Let $\phi(s,t)=(t,\gamma(s,t))$ be the evolution by isothermal gauge of the curve $\CC=(0,c)$ with initial velocity $V=(1,0,0)$. We have $\beta(s,\frac{\pi}{2})<0$ for $s\in[\frac{\pi}{2}-\eps,\frac{\pi}{2})$ and $\beta(s,\frac{\pi}{2})>0$ for $s\in(\frac{3\pi}{2},\frac{3\pi}{2}+\eps]$ for some $\eps>0$, whilst $\beta(s,\frac{\pi}{2})=0$ for $s\in[\frac{\pi}{2},\frac{3\pi}{2}]$. Moreover, $\lim_{s\to {\frac{\pi}{2}}^-} U(s,\frac{\pi}{2}) = \lim_{s\to {\frac{3\pi}{2}}^+} U(s,\frac{\pi}{2}) = (0,1)$. Thus $\gamma([\frac{\pi}{2}-\eps,\frac{3\pi}{2}+\eps],\frac{\pi}{2})$ is a $C^1$ immersed curve. See Figure \ref{cuspswitcher}(b). The numerical plot reveals some interesting geometry at the time $t_*=\frac{\pi}{2}$. We see that at this moment in time a cusp instantaneously reverses the direction of its axis, so that the spatial cross section is $C^1$ at $\phi(\frac{\pi}{2},\frac{\pi}{2})$. Although the spatial cross-section is regular at this point, the surface is not, and looks locally like a cone, with a pair of cusps tracing two ``cuts'' running down to the vertex. (One should be reminded that in this example $\frac{\pi}{2}$ is not the first time of singularity for the Cauchy evolution of $(\CC,V)$).
\begin{figure}[h]
    \centering
    	 \captionsetup{width=.9\linewidth}
    \begin{tabular}{cc}
    \begin{minipage}{0.45\textwidth}
        \centering
        \includegraphics[width=.9\textwidth]{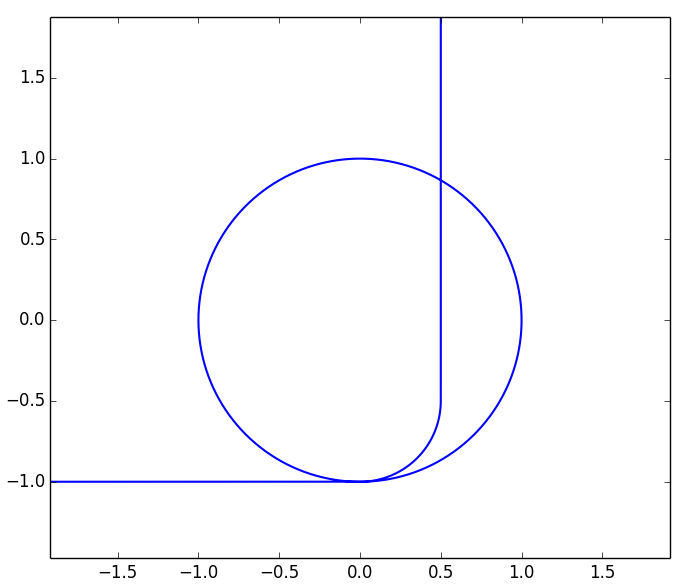}
            \end{minipage}\hfill
    &
    \begin{minipage}{0.45\textwidth}
        \centering
        \includegraphics[width=.9\textwidth]{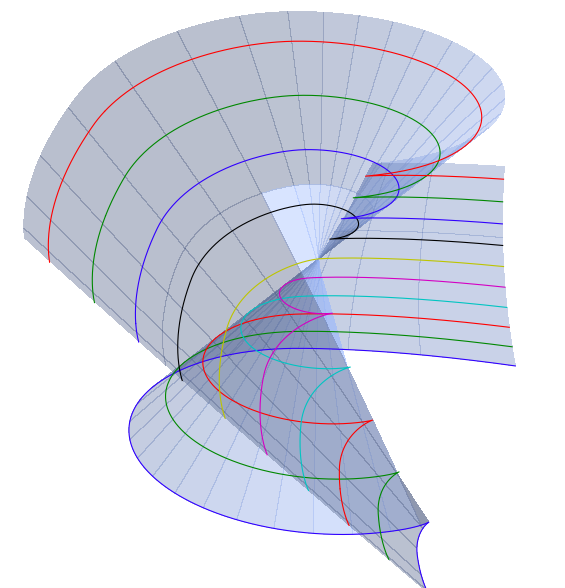}
    \end{minipage}\\
    (a) & (b)
    \end{tabular}
    \caption[An example of a `cusp reversal']{  (a) The initial curve of Example \ref{cuspswitcherexample}. (b) The evolution $\phi(s,t)$ of the curve in (a) by isothermal gauge, plotted for values $s\in[-2,10], t\in[-\pi,\pi]$. The coloured curves are $\{x^0 = \mathrm{constant}\}$ cross sections}
    \label{cuspswitcher}
\end{figure}
\end{ex}
\begin{proof}[Proof of Theorem \ref{gaugetheorem}:]
Letting $\phi(s,t)=(t,\gamma(s,t))$ be an evolution by isothermal gauge for a $C^1\times C^0$ initial data $(\CC,V)$, we are supposing that the image of the unit tangent along $\CC$ contains an arc of length $>\pi$, i.e.\ there exist $s_1, s_2\in \RR$ for which \eqref{sharpcondition} holds. By Lemma \ref{betaswitcherlemma} there exists a time $t_*\in \RR$ such that $\sin\frac{\beta(\cdot, t_*)}{2}$ takes both positive and negative values. By Lemma \ref{jerrardslemma} there exists an open interval $I$ such that for each $t\in I$ either $\Imm(\gamma(\cdot,t))$ is not a $C^1$ immersed curve or $\Imm(\gamma(\cdot,t))$ is a $C^1$ immersed curve but $U(\cdot,t)$ does not admit an extension to a continuous unit tangent vector field along $\gamma(\cdot, t)$. Theorem \ref{gaugetheorem} is proved.
\end{proof}
We conclude this section with an example where the set $\gamma([s_1,s_2],t_*)$ is a $C^1$ immersed curve, whilst $U(\cdot,t_*)$ admits no extension to a continuous unit tangent vector field along $\gamma(\cdot,t_*)$ on $[s_1,s_2]$ (thus $\gamma(\cdot, t_*)$ admits no monotone reparameterisation to a $C^1$ immersion).
\begin{ex}[Degenerate cusp singularities]\label{sheetingexample}
Consider the $C^1$ initial curve defined by
\al{c(s)= \begin{cases}
\left(- s,-1\right) &\text{for} \hskip3pt s\in(-\infty,0] \\
\left(-\sin s, -\cos s\right) &\text{for} \hskip3pt s\in (0, \frac{\pi}{2}] \\
\left(-2+\cos(s-\frac{\pi}{2}), -\sin (s-\frac{\pi}{2})\right) &\text{for} \hskip3pt s\in (\frac{\pi}{2},\pi] \\
\left(-2-\sin(s-\pi),2-\cos(s-\pi)\right) &\text{for} \hskip3pt s\in(\pi,2\pi]  \\
\left(-2+2\sin\frac{s-2\pi}{2},1+2\cos\frac{s-2\pi}{2}\right)  &\text{for} \hskip3pt s\in(2\pi,3\pi]  \\
\left(1-\cos(s-3\pi),1-\sin(s-3\pi)\right)   &\text{for} \hskip3pt s\in(3\pi,\frac{7\pi}{2}]  \\
\left(1+\sin(s-\frac{7\pi}{2}),-1+\cos(s-\frac{7\pi}{2})\right)   &\text{for} \hskip3pt s\in(\frac{7\pi}{2},4\pi]  \\
\left(2,-1-(s-4\pi)\right)     &\text{for} \hskip3pt s\in(4\pi,\infty).  \\
\end{cases}}
See Figure \ref{rivercurve}(a). Let $\phi(s,t)=(t,\gamma(s,t))$ be the evolution of $\CC(s)=(0,c(s))$ with initial velocity $V=(1,0,0)$. It may be seen that the curve $\gamma(s,\frac{3\pi}{2})=c(s+\frac{3\pi}{2})+c(s-\frac{3\pi}{2})$ will backtrack and retrace its steps twice, so that the map $s\mapsto U(s,\frac{3\pi}{2})=\nicefrac{\gamma_s(s,\frac{3\pi}{2})}{|\gamma_s(s,\frac{3\pi}{2})|}$ is discontinuous, whilst the image $\gamma([\frac{3\pi}{2},\frac{5\pi}{2}],\frac{3\pi}{2})$ is a $C^1$ curve. This phenomenon is illustrated in Figure \ref{rivercurve}(b). In this example, the degenerate behaviour is sandwiched between a pair of ordinary cusps which travel along $t=-s+2\pi, t>\pi$ and $t=s-\frac{3\pi}{2}, t>\frac{5\pi}{4}$, and the surface $\Sigma$ is not $C^1$.
\end{ex}
\begin{figure}[h]
    \centering
    	 \captionsetup{width=.9\linewidth}
    \begin{tabular}{cc}
    \begin{minipage}{0.45\textwidth}
        \centering
        \includegraphics[width=.7\textwidth]{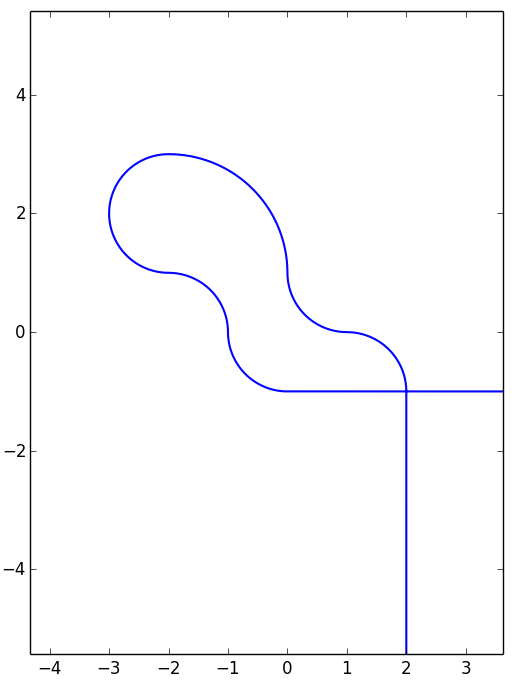}
            \end{minipage}\hfill
    &
    \begin{minipage}{0.45\textwidth}
        \centering
        \includegraphics[width=.9\textwidth]{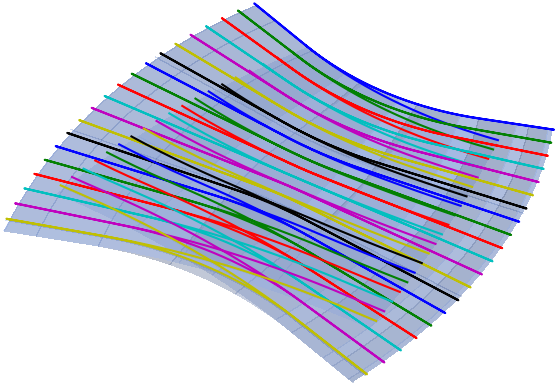}
    \end{minipage}\\
    (a) & (b)
    \end{tabular}
    \caption[An example of `sheeting']{  (a) The initial curve of Example \ref{sheetingexample}. (b) The evolution $\phi(s,t)$ of the curve in (a) by isothermal gauge, plotted for values $s\in [1.4\pi,2.6\pi], t\in[1.4\pi,1.6\pi]$. The coloured curves are $\{x^0 = \mathrm{constant}\}$ cross sections}
    \label{rivercurve}
\end{figure}
\bibliographystyle{plain}
\bibliography{embeddedness}
\end{document}